\documentclass{cams-l}

\usepackage{amsrefs}

\usepackage{epsfig}  		
\usepackage{epic,eepic}       
\usepackage{graphicx}
\usepackage{lineno}

\usepackage{enumerate}
\usepackage{mathbbol}
\usepackage{amssymb}
\usepackage{braket}
\usepackage{longtable}
\usepackage[colorlinks]{hyperref}
\usepackage{doi}
\usepackage{caption}
\usepackage{mathtools}
\usepackage{microtype}

\DeclareSymbolFontAlphabet{\mathbb}{AMSb}
\DeclareSymbolFontAlphabet{\mathbbl}{bbold}

\newenvironment{claimproof}[1][Proof of Claim]{\begin{proof}[#1]}{\end{proof}}

\newtheorem{lemma}{Lemma}[section]
\newtheorem*{lemma*}{Lemma}
\newtheorem{theorem}[lemma]{Theorem}
\newtheorem*{theorem*}{Theorem}
\newtheorem{corollary}[lemma]{Corollary}
\newtheorem{proposition}[lemma]{Proposition}

\newtheorem*{proposition*}{Proposition}
\newtheorem{fact}[lemma]{Fact}
\newtheorem*{fact*}{Fact}
\newtheorem{notation}[lemma]{Notation}
\newtheorem*{notation*}{Notation}
\newtheorem*{conventions*}{Conventions}
\newtheorem{remark}[lemma]{Remark}
\newtheorem*{remark*}{Remark}

\newtheorem*{corollary*}{Corollary}

\newtheorem*{conjecture*}{Conjecture}

\newtheorem*{problem*}{Problem}

\newtheorem*{question*}{Question}

\newtheorem{assumption*}{Assumption}

\theoremstyle{definition}

\newtheorem*{example*}{Example}
\newtheorem{definition}[lemma]{Definition}
\newtheorem*{definition*}{Definition}

\theoremstyle{remark}
\newtheorem{claim}{Claim}
\newtheorem*{claim*}{Claim}

\newtheorem*{case*}{Case}
\newtheorem*{construction*}{Construction}

\newtheorem*{exercise*}{Exercise}

\numberwithin{equation}{section}

\newcommand{\Z}{\mathbb{Z}}
\newcommand{\C}{\mathbb{C}}
\newcommand{\Q}{\mathbb{Q}}
\newcommand{\I}{\mathcal{I}}

\newcommand{\bs}{\backslash}

\newcommand{\floor}[1]{\left \lfloor #1 \right \rfloor}

\newcommand\CC{{\mathcal C}}

\newcommand\FF{{\mathcal F}}

\newcommand\II{{\mathcal I}}
\newcommand\JJ{{\mathcal J}}

\newcommand\LL{{\mathcal L}}

\newcommand\PP{{\mathcal P}}

\newcommand\<{\langle}
\renewcommand\>{\rangle}

\newcommand{\tp}{\mathrm{tp}}

\newcommand{\qftp}{\mathrm{qftp}}

\newcommand{\Ra}{\Rightarrow}

\newcommand{\La}{\Leftarrow}

\def\Ind#1#2{#1\setbox0=\hbox{$#1x$}\kern\wd0\hbox to 0pt{\hss$#1\mid$\hss}
\lower.9\ht0\hbox to 0pt{\hss$#1\smile$\hss}\kern\wd0}

\def\notind#1#2{#1\setbox0=\hbox{$#1x$}\kern\wd0
\hbox to 0pt{\mathchardef\nn=12854\hss$#1\nn$\kern1.4\wd0\hss}
\hbox to 0pt{\hss$#1\mid$\hss}\lower.9\ht0 \hbox to 0pt{\hss$#1\smile$\hss}\kern\wd0}

\def\includeE#1{{\lhook\kern-3.5pt\joinrel\smash{
		\mathop{\longrightarrow}\limits^{#1}}}}

\def\efor/{Example~\ref{E4}}

\def\BL/{Baldwin--Lachlan}
\def\Bu/{Buechler}
\def\Hr/{Hrushovski}
\def\lm/{locally modular}
\def\wm/{weakly minimal}
\def\nm/{non--modular}
\def\ss/{superstable}
\def\ud/{unidimensional}
\def\sm/{strongly minimal}

\def\abar{\bar{a}}

\def\bbar{\bar{b}}
\def\cbar{\bar{c}}
\def\dbar{\bar{d}}
\def\ebar{\bar{e}}

\def\hbar{\bar{h}}

\def\mbar{\overline{m}}
\def\nbar{\bar{n}}

\def\pbar{\bar{p}}

\def\vbar{\bar{v}}
\def\wbar{\bar{w}}
\def\xbar{\bar{x}}
\def\ybar{\bar{y}}
\def\zbar{\bar{z}}

\def\dom{{\rm dom}}

\def\tr/{trivial}
\def\nt/{non--trivial}
\def\st/{strong type}

\def\abar{\bar{a}}
\def\bbar{\bar{b}}
\def\cbar{\bar{c}}
\def\dbar{\bar{d}}
\def\ebar{\bar{e}}
\def\A{{\mathcal A}}

\def\C{{\mathfrak  C}}

\def\P{{\mathcal P}}
\def\Q{{\mathbb Q}}

\def\Z{{\mathbb Z}}

\def\dom{{\rm dom}}

\def\range{{\rm range}}

\def\Fa0{{\FF^a_{\aleph_0}}}

\def\<{\langle}
\def\>{\rangle}

\newcommand\myrestriction{\mathord\restriction}
\def\mr#1{\myrestriction_{#1}}

\renewcommand{\C}{\mathfrak{C}}

\def\rtp{{\rm rtp}}

\def\qftp{{\rm qftp}}

\def\P{{\mathcal P}}

\newcommand\e{e}
\newcommand\eprec{\prec_\exists}
\newcommand\etp{{\rm tp_\exists}}

\newcommand\ertp{{\rm rtp}_{\exists}}

\title{Existential characterizations of monadic NIP}
\author{Samuel Braunfeld}
\address{Computer Science Institute, Charles University. Prague, Czechia, and The Czech Academy of Sciences, Institute of Computer Science, Pod Vod\'{a}renskou v\v{e}\v{z}\'{\i} 2, 182 00 Prague, Czech Republic.}
\thanks{This paper is part of a project that has received funding from the 
European Research Council (ERC) under the European Union's Horizon 2020 
research and innovation programme (grant agreement No 810115 - Dynasnet). Samuel Braunfeld is further supported by Project 24-12591M of the Czech Science Foundation (GA\v{C}R), and supported partly by the long-term strategic development financing of the Institute of Computer Science (RVO: 67985807).}

\author{Michael C. Laskowski}
\address{University of Maryland, College Park. 4176 Campus Dr., College Park, MD 20742}
\thanks{Michael Laskowski is partially supported by NSF grants DMS-1855789 and DMS-2154101.}

\begin{document}
	\begin{abstract}
	We show that if a universal theory is not monadically NIP, then this is witnessed by a canonical configuration defined by an existential formula. As a consequence, we show that a hereditary class of relational structures is NIP (resp. stable) if and only if it is monadically NIP (resp. monadically stable). As another consequence, we show that if such a class is not monadically NIP, then it has superexponential growth rate.
\end{abstract}
\maketitle

\section{Introduction}
A common theme in both logic and combinatorics is analyzing the complexity of classes of structures, and identifying jumps in the spectrum of possible complexities. For illustration, consider (the asymptotics of) the function counting the number of structures of a given size in the class as one such complexity measure. Surveys concerning this counting problem in hereditary classes of finite structures, i.e. those closed under taking substructures,  include \cites{bollobas12007hereditary, klazar2010some}. The analogue of this problem in the case of the class of models of a complete first-order theory and where the sizes considered are infinite guided the discovery and development of many of the fundamental concepts of model theory \cite{shelah1990classification}. Prominent among these model-theoretic concepts, in increasing strength, are NIP, stability, and NFCP; these are properties of a class of structures defined by excluding certain combinatorial configurations that serve as canonical witnesses to complicated behavior. Model theorists have developed several tools to analyze theories with these properties. Although these tools primarily apply to infinite structures, one can extend a class of finite structures to include infinite structures, apply the tools to the infinite structures in the extended class, and then try to pull the results down to the finite.

When studying the combinatorial behavior of the models of a complete theory, NFCP, stability, and NIP repeatedly appear as dividing lines between tame and wild behavior. Recently their monadic variants, which strengthen these properties by requiring that they remain true under arbitrary coloring of the points of the structures in any number of colors, have seen many applications in hereditary classes of finite structures. For examples concerning the asymptotics of the counting function,  \cite{laskowski2022jumps} shows that monadic NFCP (there called mutual algebraicity) plays a key role in identifying a jump in the possible asymptotics, while \cite{braunfeld2022monadic} shows that monadic stability explains the jump from sub-exponential to exponential growth in a setting corresponding to orbit-counting for certain group actions, and \cite{bonnet2024twin} shows that monadic NIP explains the jump from exponential growth to factorial-type growth for classes of ordered graphs. One of the results of this paper (Theorem \ref{thm:intro growth}) extends one direction of this last result beyond ordered graphs, showing that if a class of relational structures is not monadically NIP then it has at least factorial-type growth. Further applications of these properties in the finite setting, including to questions of algorithmic tractability, are surveyed in \cite{pilipczuk2025graph}.

The appearance of \emph{monadic} NFCP/stability/NIP in the results above is curious, and likely unexpected to a model theorist. After all, the monadic aspect is concerned with coloring the points of structures, but such colorings play no role in the statements or their proofs. The main result of this paper provides an explanation for this: in hereditary classes of relational structures, stability and NIP are equivalent to their monadic variants. The monadic versions of these properties appear as dividing lines in hereditary classes because there they agree with their non-monadic counterparts.

\begin{theorem}[Theorem \ref{thm:coll}, Theorem \ref{thm:coll mon}] \label{thm:intro coll}
	Let $\P \in \set{\textrm{stable, NIP}}$. Let $\CC$ be a hereditary class of relational structures. The following are equivalent.
	\begin{enumerate}
		\item $Th(\CC)_{\forall}$ is monadically $\P$.
		\item $Th(\CC)$ is monadically $\P$.
		\item $Th(\CC)_{\forall}$ is $\P$.
		\item $Th(\CC)$ is $\P$.
	\end{enumerate}
	
	Furthermore, if $\CC$ is monotone, then $Th(\CC)$ is NIP if and only if $Th(\CC)$ is monadically stable.
\end{theorem}

Since the models of a monadically NIP theory have a tree-like structure (see the characterization in \cite{pwidth}, while the characterizations in \cite{MonNIP} even show an order-like structure if we look coarsely enough), this theorem suggests that we should expect a sharp dichotomy in hereditary classes: either the structures admit tree-like decompositions, or the class is at least as complicated as the class of all graphs. This dichotomy may also partially explain the ubiquity of tree-like decompositions in structural graph theory.

The key step behind Theorem \ref{thm:intro coll} is that if a universal theory is not monadically NIP, then this is witnessed by a configuration defined by an {\em existential} formula.

\begin{theorem} [Theorem \ref{thm:e pre-code}] \label{thm:intro pre}
	For a universal theory $T$, the following are equivalent.
	\begin{enumerate}
		\item $T$ is monadically NIP.
		\item $T$ does not admit pre-coding by an existential formula.
		\item $T$ has the $\e$-f.s. dichotomy
	\end{enumerate}
	
	In particular, if $T$ is not monadically NIP then it admits pre-coding by an existential formula.
\end{theorem}

While \cite{MonNIP} showed a pre-coding configuration must definably appear in theories that are not monadically NIP, the existential condition allows for much greater control and more direct finitization, and we expect further applications of this result, as in \cite{horizons}.

To illustrate that Theorem \ref{thm:intro pre} makes (the failure of) monadic NIP manifest within standardly-considered combinatorics, we finish by using it to show that if a hereditary class $\CC$ is not monadically NIP then it has superexponential growth rate.

\begin{theorem} [Theorem \ref{thm:growth}] \label{thm:intro growth}
	Let $\CC$ be a hereditary class of relational structures. If $Th(\CC)$ is not monadically NIP, then there is some $k \in \omega$ such that the unlabeled growth rate $f_\CC(n) = \Omega(\floor{n/k}!)$. 
\end{theorem}

After preliminary material in Section \ref{sec:prelim}, Theorem \ref{thm:intro pre} is proved in Section \ref{sec:precode}. The proof recapitulates some of the material of \cite{MonNIP} in the setting of existentially closed models of a universal theory, but must ultimately return to the setting of saturated models and connect to the characterizations given in \cite{MonNIP}. Section \ref{sec:precode} also shows that if a universal theory is monadically NIP but not monadically stable, then there is an atomic formula witnessing the order property. We prove Theorem \ref{thm:intro coll} in Section \ref{sec:app}, by manipulating a generalized indiscernible instance of the configuration provided by Theorem \ref{thm:intro pre}. Further manipulations of this configuration in Section \ref{sec:growth} yield Theorem \ref{thm:intro growth}.

\subsection{Acknowledgments} We thank Gregory Cherlin for suggesting the proof of Lemma \ref{l:ind ext}. We thank the referees for their suggestions greatly improving the clarity of presentation.

\section{Preliminaries} \label{sec:prelim}

Model theory typically takes place in the category whose objects are models of a complete theory $T$ and whose morphisms are elementary embeddings. 
For understanding models of a complete theory, it  convenient to work  in a large saturated ``monster model''  $\C$ of $T$.  
Throughout the paper, 
when we are proving something about a complete theory $T$, we will implicitly assume that we are working within such a $\C$.
In particular, 
all tuples and sets that we consider will be `small subsets'
of $\C$ and all models $M$ will be elementary submodels of $\C$. 

By contrast, in many places we will need to analyze models of {\em universal} theories $T$ that are not complete.   In those settings, it is 
useful to  work in the category whose objects are existentially closed models of $T$ and whose morphisms are $\e$-embeddings, i.e., preserve existential formulas.

To motivate our interest in  universal theories, we begin by showing that for monadic properties, only the universal part of a theory is relevant. 

\begin{definition} \label{def:prop}
	We will say an incomplete theory $T$ is (monadically) stable/NIP if this is true for all completions of $T$.
\end{definition}

\begin{notation}
	Given a theory $T$, we let $T_{\forall}$ denote the set of universal sentences that are logical consequences of  $T$.
\end{notation}

The following  semantic equivalent of $T_\forall$ will be used repeatedly.

\begin{fact} \label{f:univ sub}
	Let $T$ be an $\LL$-theory. Then  for any $\LL$-structure $M$, 
	$M \models T_{\forall}$ if and only if there is some $N \models T$ such that $M \subset N$.
\end{fact}

\begin{lemma} \label{l:mon sub}
	Let $\P \in \set{\textrm{stable, NIP}}$. Suppose $N$ is monadically $\P$ and $M \subset N$. Then $M$ is monadically $\P$.
\end{lemma}
\begin{proof}
	If $M$ is not monadically $\P$, then there is some unary expansion $(M, U_1, \dots, U_k)$ containing arbitrarily large finite approximations to the configuration witnessing the failure of $\P$, uniformly defined by $\phi$. Let $U$ be a unary predicate such that $U(N) = M$, interpret $U_i(N) = U_i(M)$ for every $i \in [k]$, and let $\phi'$ be obtained by restricting all quantifiers in $\phi$ to $U$. Then $\phi'$ also defines arbitrarily large finite approximations to the configuration witnessing $\P$ in $(N, U, U_1, \dots, U_k)$.
\end{proof}

\begin{proposition} \label{p:mon univ}
	Let $\P \in \set{\textrm{stable, NIP}}$. Let $T$ and $T'$ be theories, such that $T_{\forall} \subseteq T'_{\forall}$. If $T$ is monadically $\P$, then $T'$ is monadically $\P$.
	
	In particular,  $T$ is monadically $\P$ if and only if $T_{\forall}$ is monadically $\P$.
\end{proposition}
\begin{proof}
	Fix $M' \models T'$. Since $T_{\forall} \subseteq T'_{\forall}$, $M' \models T_{\forall}$, so by Fact \ref{f:univ sub} there is $M \models T$ such that $M' \subset M$. As $M$ is monadically $\P$, so is $M'$ by Lemma \ref{l:mon sub}.
\end{proof}

Looking ahead, at least for theories in relational languages, Corollary~\ref{stengthen2.5} will significantly strengthen this result.  

\subsection{Existentially closed models} \label{sec:ec}

When working with universal theories $T$, it is frequently useful to restrict attention to the set of existential formulas,  i.e., formulas $T$-equivalent to formulas
of the form $\exists\ybar\delta(\xbar,\ybar)$,
where $\delta$ is quantifier free. The initial notions of this section are standard, and \cite[Chapter 8.1-8.2]{hodges1993model} can serve as a reference for our facts stated without proof.   

\begin{definition}
	Let $T$ be a universal theory. 
	\begin{itemize}  
	\item  For $M\models T$ and $A,\mbar$ from $M$, $\etp^M(\mbar/A)=\{\hbox{existential}\ \phi(\xbar,\abar):M\models \phi(\mbar,\abar),\abar\subset A\}$.  
	When $M$ is understood, we simply write $\etp(\mbar/A)$ and when $A$ is empty we write $\etp(\mbar)$.
	\item For $M\subset N$, both models of $T$, we write $M\prec_\exists N$ if $M\models\phi(\mbar)$ if and only if $N\models\phi(\mbar)$ for every existential formula
	$\phi(\xbar)$ and every $\mbar\in M^{\lg(\xbar)}$.  
	
	\item  We say  $M \models T$ is {\em existentially closed}   if for any $\mbar \in M$ and existential $\phi(\xbar)$, if there is some $N \supset M$ such that $N \models T$ and $N \models \phi(\mbar)$, then already $M \models \phi(\mbar)$.
	
	\end{itemize}
\end{definition}

Trivially, if $M\subset N$ are models of $T$,  then $\etp^M(\mbar/A)\subset \etp^N(\mbar/A)$, but in general, equality need not hold.  
For an arbitrary $M\models T$ and $\mbar$ from $M$, we say {\em $\etp(\mbar)$ is maximal existential} if $\etp^N(\mbar)=\etp^M(\mbar)$ for every $N\supseteq M$ with $N\models T$.  To understand this, define an {\em obstruction to an existential formula $\phi(\xbar)$} to be an existential formula $\psi(\xbar)$ for which $T\models\forall\xbar [\psi(\xbar)\rightarrow \neg\phi(\xbar)]$.

The following Fact, which shows that a model of $T$ being existentially closed can be construed as an omitting types statement,
 is easily proved by  unpacking the definitions and compactness.

\begin{fact}  The following are equivalent for a model $M$ of a universal theory $T$.
\begin{enumerate}
\item  $M$ is existentially closed;
\item  $\etp(\mbar)$ is maximal for every finite $\mbar$ from $M$;
\item  For every existential formula $\phi(\xbar)$, $M$ omits the type
$\Gamma_\phi(\xbar):=\{\neg\phi(\xbar)\}\cup\{\neg\psi(\xbar):\ \hbox{$\psi(\xbar)$ is an obstruction to $\phi(\xbar)$}\}$.
\end{enumerate}
\end{fact}

\begin{fact}  \label{ecexists}
Let $T$ be a universal theory. 
\begin{enumerate}
\item
For every $M \models T$, there is $N \supset M$ such that $N$ is an existentially closed model of $T$ and $|N|=|M| + |T|$.
\item  The class of existentially closed models of $T$ is closed under unions of chains.
\end{enumerate}
\end{fact}

Life inside an existentially closed model of $T$ is pleasant.

\begin{fact} \label{insideec}    Let $T$ be universal and $N$  an existentially closed model of $T$.
\begin{enumerate}
\item  For any $M\subset N$, $M$ is an existentially closed  model of $T$ if and only if  $M\prec_\exists N$.
\item  For any subset $B\subset N$, there is an existentially closed model $M$ of $T$ with $B\subset M\subset N$ and $|M|=|B|+|T|$.
\end{enumerate}
\end{fact}

We describe  some `large' existentially closed models of $T$.  

\begin{definition}  Let $T$ be a universal theory and let $\kappa\ge|T|$ be a cardinal.  A model $M\models T$ is {\em $\kappa^+$-existentially closed } if
$M$ is existentially closed and for every set $\Gamma(\xbar)$ of existential $L(M)$-formulas with $|\Gamma(\xbar)|\le \kappa$, if
$\Gamma(\xbar)$ is finitely satisfied in $M$, then there is some $\bbar\in M^{\lg(\xbar)}$ realizing $\Gamma(\xbar)$.
\end{definition}  

One should think of a $\kappa^+$-existentially closed model of $T$ as being an approximation to  a $\kappa^+$ saturated model, but only in regard to existential types over subsets of $M$ of size at most $\kappa$.
Note that if $M$ is $\kappa^+$-existentially closed, then $|M|>\kappa$.  Also, since each of the relevant $\Gamma(\xbar)$ has size at most $\kappa$, $|\lg(\xbar)|\le \kappa$.  

\begin{lemma} \label{onekappa}  Suppose $T$ is universal and $\kappa\ge |T|$.  For every existentially closed model\ $M\models T$ with $|M|=2^\kappa$, there is an existentially closed model $N\supseteq M$, 
with $|N|=2^\kappa$ that realizes every set $\Gamma(\xbar)$ of existential $L(M)$-formulas with $|\Gamma(\xbar)|\le \kappa$ that is finitely satisfied in $M$.
\end{lemma}

\begin{proof}  There are only $2^\kappa$ such $\Gamma(\xbar)$, hence there is an elementary extension $M'\succeq M$ of size $2^\kappa$ realizing all of these.
$M'$ need not be an existentially closed model, but we can apply Fact~\ref{ecexists} to get such an $N\supset M'$.
\end{proof}

\begin{lemma}  \label{kappaecexists}
 Suppose $T$ is universal, $M\models T$ and $\kappa\ge|M|+|T|$.   Then there is a $\kappa^+$-existentially closed model  $N\models T$ with $N\supset M$ and $|N|=2^\kappa$.
\end{lemma}

\begin{proof}  Using Fact~\ref{ecexists} choose an existentially closed $M_0\models T$ of size $2^\kappa$ containing $M$.  Then iterate Lemma~\ref{onekappa} $\kappa^+$ times,
taking unions at limit ordinals.  
\end{proof}  

We see that $\kappa^+$-existentially closed models of $T$ are rather homogeneous.  Under the assumption that $2^\kappa=\kappa^+$, a $\kappa^+$-existentially closed model of size $2^\kappa$ is fully homogeneous, but without it, we restrict to extending a single $\e$-map (defined below) to an automorphism.

\begin{definition}  Let  $T$ be universal, $\kappa\ge |T|$, and   $M$  a $\kappa^+$-existentially closed model of $T$.
\begin{itemize}
\item  An {\em $\e$-map} is a bijection $f:A\rightarrow B$ with $|A|=|B|\le\kappa$ satisfying $\etp(A)=\etp(B)$, i.e., $\etp(\abar)=\etp(f(\abar))$ for all $\abar$ from $A$.
\item  For $C\subseteq M$, an {\em $\e$-map inside $C$} is an $\e$-map $f:A\rightarrow B$ with $A\cup B\subseteq C$.  
\item  An {\em $\e$-permutation of $C$} is an $\e$-map $f:C\rightarrow C$.
\end{itemize}
\end{definition}

Note that if $N\models T$, then any $\e$-permutation of $N$ is an automorphism.

\begin{lemma}  \label{twopart}  Let $T$ be universal, $\kappa\ge |T|$, and $M$  a $\kappa^+$-existentially closed model of $T$.
\begin{enumerate}
\item  For every $\e$-map $f:A\rightarrow B$ and every $C\subset M$ with $|C|\le\kappa$, there is $D\subset M$ and an $\e$-map $g:AC\rightarrow BD$ extending $f$.
\item  For every $C\subset M$ with $|C|\le \kappa$ and for every $\e$-map $f:A\rightarrow B$ inside $C$, there is $D$, $C\subset D\subset M$, $|D|\le\kappa$, and
an $\e$-permutation $g:D\rightarrow D$ extending $f$.
\end{enumerate}
\end{lemma}

\begin{proof}  (1)  Fix an enumeration $\cbar$ of $C$ and let
$$\Gamma(\xbar,B):=\{\phi(\xbar',f(\abar)):\phi(\xbar',\abar)\in \etp(\cbar'/A)\}$$
Since $\etp(C/A)$ is finitely satisfied in $M$ and $\etp(A)=\etp(B)$, $\Gamma(\xbar,B)$ is finitely satisfied in $M$, so take $D$ to be any realization of $\Gamma(\xbar,B)$.

(2)  We construct a nested  $\omega$-sequence $(C_n:n\in\omega)$ of subsets of $M$, all of size $\le\kappa$, and a nested sequence $g_n$ of $\e$-maps inside $C_n$ as follows:
Put $C_0=C$ and $g_0=f$.   Given $C_n$ and an $\e$-map $g_n$ inside $C_n$, apply (1) twice for both the domain and range to get
$C_{n+1}\supseteq C_n$ with $|C_{n+1}|=\kappa$. and an $\e$-map $g_{n+1}$ inside $C_{n+1}$ extending $g_n$ with $C_n\subseteq \dom(g_{n+1})\cap\range(g_{n+1})$.
Then $D:=\bigcup C_n$ and $g:=\bigcup g_n$ work.
\end{proof}

\begin{lemma}  \label{auto}  Let $T$ be universal, $\kappa\ge |T|$, and $M$  a $\kappa^+$-existentially closed model of $T$.  For all $A,B,C\subseteq M$ with $|ABC|\le\kappa$
and for every $\e$-map $f:A\rightarrow B$,
there is an $N\prec_\exists M$, $|N|=\kappa$, with $ABC\subset N$ and $\sigma\in Aut(N)$ extending $f$.
\end{lemma}

\begin{proof}  Build a nested sequence $N_0\subset D_0\subset N_1\subset D_1\subset \dots$ and a nested sequence of $\e$-permutations $g_n:D_n\rightarrow D_n$ 
as follows.   By Fact~\ref{insideec}(2), choose an existentially closed model $N_0$
satisfying $ABC\subset N_0\subset M$ with $|N_0|=\kappa$.   Get $D_0$ and an $\e$-permutation $g_0:D_0\rightarrow D_0$ extending $f:A\rightarrow B$
from Lemma~\ref{twopart}(2), taking $C$ there to be $N_0$.  Continuing, given $D_n$ and $g_n:D_n\rightarrow D_n$, apply Fact~\ref{insideec}(2) to get $N_{n+1}$ of size $\kappa$ such that  $D_n\subset N_{n+1}\subset M$.  Then obtain $D_{n+1}$ and an $\e$-permutation $g_{n+1}$ of $D_{n+1}$ from  Lemma~\ref{twopart}(2), taking $C$ to be $N_{n+1}$
and $g_n$ to be the given $\e$-map.  

 To finish, put $N:=\bigcup N_n$.  Then $N$ is an existentially closed model of $T$ by Fact~\ref{ecexists}(2) and $g:=\bigcup g_n$ is an automorphism of $N$ extending $f$.
\end{proof}

Finally, as we will make significant use of finite satisfiability, we review some of its properties in the existentially closed setting.

\begin{lemma} \label{lem:basic}
	Let $M\eprec N$, with $N$ existentially closed.
	\begin{enumerate}
		\item Every maximal existential type $p$ over $M$ is finitely satisfiable in $M$.
		\item  (Non-$\e$-splitting)  If $p$ is a maximal existential type over $B \subset N$ that is finitely satisfied in $M$, then $p$ does not $\e$-split over $M$, i.e., if $\bbar,\bbar'\subseteq B$ and $\etp(\bbar/M)=\etp(\bbar'/M)$,
		then for any $\phi(\xbar,\ybar)$, we have $\phi(\xbar,\bbar)\in p$ if and only if $\phi(\xbar,\bbar')\in p$.
		\item  (Transitivity)  Let $\abar, \bbar, C \subset N$. If $\etp(\bbar/C)$ and $\etp(\abar/\bbar C)$ are both finitely satisfied in $M$, then so is $\etp(\abar\bbar/C)$.
	\end{enumerate}
\end{lemma}
\begin{proof}
	$(1)$ As maximal existential types are closed under conjunctions, we may consider a single formula $\phi(\xbar; \mbar) \in p$, with $\mbar \subset M$. By definition of being existentially closed, there is some $\nbar \in M$ such that $M \models \phi(\nbar; \mbar)$.
	
	$(2)$ Suppose $\phi(\xbar, \bbar) \in p(\xbar)$, but $\phi(\xbar, \bbar') \not\in p(\xbar)$. Then there is $\psi(\xbar, \cbar) \in p(\xbar)$ that is an obstacle for $\phi(\xbar, \bbar')$. By finite satisfiability, there is some $\mbar \in M$ such that $N \models \phi(\mbar, \bbar) \wedge \psi(\mbar, \cbar)$, so $N \models \neg \phi(\mbar, \bbar')$, and so $\etp(\bbar/M) \neq \etp(\bbar'/M)$.
	
	$(3)$ Let $\phi(\xbar, \ybar; \cbar) \in \etp(\abar \bbar/C)$. Since $\etp(\abar/\bbar C)$ is finitely satisfied in $M$, there is $\mbar \subset M$  such that $N \models \phi(\mbar, \bbar; \cbar)$. Then since $\etp(\bbar/C)$ is finitely satisfied in $M$, there is $\mbar' \subset M$ such that $N \models \phi(\mbar, \mbar'; \cbar)$.
\end{proof}

One of the key facts about finitely satisfiable types in the usual setting of complete theories is that they are precisely the average types of ultrafilters. From this it follows that if $p$ is a type over $A$ that is finitely satisfied in $M$, then for any $B \supset A$, there is an extension of $p$ to a type over $B$ that is still finitely satisfied in $M$. However, the average existential type of an ultrafilter need not be maximal existential, and so both of these facts fail in the existentially closed setting of this subsection. Because of this, we will have to leave the setting of existentially closed models for some arguments.

\subsection{Generalized indiscernibles indexed by ordered pairing functions}

Our later results giving consequences of a pre-coding configuration are aided by generalized indiscernibles indexed by the following structure. This is essentially a three-sorted structure, with two sorts consisting of copies of $(\Q, \leq)$ and the third sort consisting of the product $(\Q^2, \leq_{lex})$ with the lexicographic order, equipped with projection functions for each coordinate and a pairing function.

\begin{definition}  Let $L_0=\{I,J,\Gamma,\pi_1,\pi_2,\rho,\le\}$ and let $\P$ denote the countable $L_0$-structure satisfying:
\begin{enumerate}
\item  $I,J,\Gamma$ partition $\P$ into three (disjoint) sorts, each infinite;
\item  $\le$ is a dense linear order without endpoints with $I\ll J\ll \Gamma$ (so that the restriction to each sort is also dense without endpoints);
\item  $\pi_1:\Gamma\rightarrow I$, $\pi_2:\Gamma\rightarrow J$, and $\rho:I\times J\rightarrow \Gamma$;
\item  For $\gamma\in \Gamma$, $\rho(\pi_1(\gamma),\pi_2(\gamma))=\gamma$; and
\item  For $i\in I,j\in J$,  $\pi_1(\rho(i,j))=i$ and $\pi_2(\rho(i,j)=j$.
\item  $\le\mr{\Gamma}$ is $\le_{lex}$, i.e., $\gamma\le\gamma'$ iff 
either $\pi_1(\gamma)<\pi_1(\gamma')$ or [$\pi_1(\gamma)=\pi_1(\gamma')$ and $\pi_2(\gamma)\le\pi_2(\gamma')$.]
\end{enumerate}
\end{definition}

Note that $(I,\le)$ and $(J,\le)$ are each isomorphic to $(\Q,\le)$ and are order-indiscernible sequences in $\P$.  Moreover, 
 the automorphism group $Aut(\P)$ is naturally isomorphic to $Aut(I,\le)\times Aut(J,\le)$.  Indeed, for any $\sigma\in Aut(I,\le)$ and $\tau\in Aut(J,\le)$,
 we obtain an $L_0$-elementary bijection of $\Gamma$ via $\rho(i,j)\mapsto \rho(\sigma(i),\tau(j))$.

Additionally, $\P$ is uniformly locally finite.  As notation, for a finite, non-empty $X\subseteq\P$, let $\<X\>\in Age(\P)$ be the smallest substructure containing $X$.
As $\P$ is totally ordered, for any finite  substructures $A,B\subseteq\P$, there is at most one isomorphism $h:A\rightarrow B$. 
If $\delta_1<\dots<\delta_k$, $\delta_1'<\dots<\delta_k'$ are from $\P$,  then $\qftp^{\P}(\delta_1,\dots,\delta_k)=\qftp^{\P}(\delta_1',\dots,\delta_k')$ if and only if
the substructures $\<\delta_1,\dots,\delta_k\>$ and $\<\delta_1',\dots,\delta_k'\>$ are isomorphic.  

The following lemma demonstrates that $\P$ is a desirable index structure. We refer to the {\em Ramsey property} in the sense of structural Ramsey theory; see, for example, \cite[Definition 3.6]{Scow} for a definition.

\begin{lemma}   \label{ageRamsey} $Age(\P)$ has the Ramsey property and $\P$ is its Fra\"isse limit.
\end{lemma}
\begin{proof}
	That $\P$ is homogeneous follows from the natural isomorphism of $Aut(\P)$ with $Aut(I,\le)\times Aut(J,\le)$. Similarly, by associating a finite substructure of $\P$ with the product of its projections to $(I, \le)$ and $(J, \le)$, the Ramsey property becomes an instance of the product Ramsey theorem \cite[\S 5.1, Theorem 5]{GRS}.
\end{proof}
 
As $\P$ is a Fra\"isse limit, any $L_0$-isomorphism of finite substructures of $\P$ extends to an automorphism of $\P$.  

\begin{definition}  Suppose $\LL$ is any language and $M$ is any $\LL$-structure.  A subset $\A\subseteq M$ is {\em $\P$-indiscernible via $f$} if $f:\P\rightarrow M^{<\omega}$ and
$\A=\bigcup f(\P)$ satisfies:
$$\tp_L(f(\delta_1),\dots f(\delta_k))=\tp_L(f(\delta_1'),\dots,f(\delta_k'))$$
for all $k$ and all
$(\delta_1,\dots,\delta_k)$, $(\delta_1',\dots,\delta_k')$ from $\P^k$ satisfying the same quantifier free type in $\P$.

We routinely write $\{\abar_i:i\in I\}$, $\{\bbar_j:j\in J\}$, and $\{\cbar_{i,j}:(i,j)\in I\times J\}$ as the images under $f$.  
We say $\A$ is a {\em $\P$-indiscernible partition} if 
 $f(\delta)$ is without repetition and $f(\delta)$ and $f(\delta')$ are disjoint
for distinct $\delta,\delta'\in\P$.  
\end{definition}

Combining our remarks above, an equivalent of the indiscernibility condition is that for any automorphism $\sigma\in Aut(\P)$, we have that
$\tp_L(f(\delta_1),\dots,f(\delta_k))=\tp_L(f(\sigma(\delta_1)),\dots,f(\sigma(\delta_k)))$ for all $(\delta_1,\dots,\delta_k)\in \P^k$.
Thus, if in addition, $M$ is uncountable and saturated, this condition is also equivalent to:
  For every $\sigma\in Aut(\P)$, the set mapping $f(\delta)\mapsto f(\sigma(\delta))$ extends to
an automorphism $\sigma^*\in Aut(M)$.

We now record some technical remarks that will be useful in later sections. For these remarks, we work in an ambient monster model $\C$ of some fixed complete theory. 

\begin{remark}  \label{addstrips} {\em 
If $\A$ is a $\P$-indiscernible partition then there are constants $m_a,m_b,m_\Gamma$ such that $|f(i)|=m_a$, $|f(j)|=m_b$, and $|f(\gamma)|=m_\Gamma$ for all
$i\in I$, $j\in J$, and $\gamma\in\Gamma$.  Thus, if $\A$ is a  $\P$-indiscernible partition, then $\A$ is partitioned into $m_a+m_b+m_\Gamma$ `strips', indexed by
$I$, $J$, or $I\times J$.  Let $\C^+$ denote the monadic expansion of $\C$ formed by adding unary predicates $U_\ell$ for each of these strips.
Note that because of our equivalent formulation of $\P$-indiscernibility in terms of automorphisms, it follows that $\A$ is also $\P$-indiscernible in the expanded
language $\LL^+$ of $\C^+$. } 
\end{remark}

\begin{remark}  \label{P0}  {\em 
If  $I_0\subseteq I$ and $J_0\subseteq J$ are each infinite convex subsets, then the substructure $\P_0\subseteq \P$
with universe $I_0\cup J_0\cup\{\gamma_{i,j}:i\in I_0,j\in J_0\}$ is isomorphic to $\P$ and every automorphism of $\P_0$ extends to an automorphism of $\P$.
Thus, up to re-indexing, we may view $\A_0:=f[\P_0]$ as also $\P$-indiscernible via $f$.  
Furthermore, every automorphism of $\P_0$ extends to an automorphism of $\P$ fixing
$(I\setminus I_0)\cup(J\setminus J_0)$ pointwise.  In this case $\A_0$ will be $\P$-indiscernible in any expansion of $\C$ formed by naming constants from
$\{\abar_i:i\in I\setminus I_0\}\cup\{\bbar_j:j\in J\setminus J_0\}$.}
\end{remark}

\begin{remark}  \label{duplication}  {\em Suppose $\A$ is a $\P$-indiscernible partition.  As the linear order $(2\times\Q,\le)$ embeds into $(I,\le)$,
choose disjoint, dense $I_0,I_1\subseteq I$ such that for each $i\in I_1$ there is a unique $i^-\in I_0$ that is an immediate predecessor of $i$ in $I_0 \cup I_1$ and dually,
for each $i^-\in I_0$ there is a unique $i\in I_1$ that is an immediate successor in $I_0 \cup I_1$.  For each $i\in I_1$, put $\abar^*_i:=\abar_{i^-}\abar_i$, where $i^-$ is the immediate predecessor
of $i$.  Then $\A'=\{\{\abar^*_i:i\in I_1\},\{\bbar_j:j\in J\},\{\cbar_{i,j}:i\in I_1,j\in J\}\}$ is a $\P_1$-indiscernible partition, where $\P_1$ is indexed by $I_1,J,I_1\times J$.
By symmetry, one can do the same procedure on the $J$-side as well.}
\end{remark}

\section{Configurations defined by low-complexity formulas} \label{sec:precode}

\subsection{Monadic NIP} \label{sec:monNIP}

We now aim to prove Theorem \ref{thm:intro pre}. The first part of the following definition from \cite{Hanf} is the central characterization of monadically NIP theories in \cite{MonNIP}. (We remark that we will often abuse notation and treat tuples as unordered sets, as in the following definition).

\begin{definition}
	A theory $T$ has the {\em f.s.\ dichotomy} if for every $N \models T$, $M \prec N$, $\abar, \bbar \subset N$ such that $\tp(\bbar/M\abar)$ is finitely satisfiable in $M$, and $c \in N$, we have either $\tp(\bbar c/M\abar)$ or $\tp(\bbar/M\abar c)$ is finitely satisfiable in $M$.
	
	A universal theory $T$ has the {\em $\e$-f.s.\ dichotomy} if for every existentially closed $N \models T$, $M \eprec N$, $\abar, \bbar \subset N$ such that $\etp(\bbar/M\abar)$ is finitely satisfiable in $M$, and $c \in N$, we have either $\etp(\bbar c/M\abar)$ or $\etp(\bbar/M\abar c)$ is finitely satisfiable in $M$.
\end{definition}

The f.s.\ dichotomy shows that if we have two independent sets (in the sense of finite satisfiability), then we may extend this by an arbitrary point and maintain independence. Iterating these extensions allows us to produce decompositions of models of a monadically NIP theory into independent pieces as in \cite[Section 3.1]{MonNIP}, and as in Lemma \ref{l:e extend} in the existentially closed setting. Such decompositions have certain controlled type-counting, which will characterize monadic NIP in Proposition \ref{p:blobwidth}. On the other hand, if the f.s.\ dichotomy fails, then there is a point $c$ behaving ``non-generically'' with respect to the independent tuples $\abar$ and $\bbar$. This ``pairing function''-type behavior can be used to produce the pre-coding configuration of Definition \ref{def:pre-code} that is the canonical obstruction to monadic NIP.

Our first subgoal is to show that if $T$ has the $\e$-f.s.\ dichotomy, then $T$ is monadically NIP. We begin by introducing a characterization of monadic NIP considering arbitrary models of a theory. Then we prove a parallel result in the existentially closed setting assuming the $\e$-f.s.\ dichotomy, and finally show how to link them.

\begin{definition}
	Given $M \eprec N$ existentially closed models of a theory $T$ (resp. $M \prec N$ models of $T$), an  {\em $\e$-$M$-f.s.\ sequence} (resp. {\em $M$-f.s.\ sequence}) is a sequence of sets $(A_i \subset N: i \in I)$ (with $I$ linearly ordered) such that $\etp(A_i/MA_{<i})$ (resp. $\tp(A_i/MA_{<i})$) is finitely satisfiable in $M$, for every $i \in I$.
	
	Suppose $X \subseteq N$ is any set. 
	\begin{itemize}
		\item   A {\em partial $\e$-$M$-f.s.\ decomposition of $X$} is an $\e$-$M$-f.s.\ sequence $(A_i:i\in I)$ with $\bigcup_{i\in I} A_i\subseteq X$.
		\item  An {\em $\e$-$M$-f.s.\ decomposition of $X$} is a partial $\e$-$M$-f.s.\ decomposition with $\bigcup_{i\in I}A_i=X$.
	\end{itemize}

\end{definition}

\begin{definition}
	For any structure $N$ and $A,B\subseteq N$, let {\em $\rtp(B,A)$} (resp. {\em $\rtp_{qf}(B,A)$}, $\rtp_{\exists}(B, A)$) denote the number of complete types (resp. quantifier-free types, existential types) over $A$ {\em realized} by tuples in $(B\setminus A)^{<\omega}$.
\end{definition}

\begin{definition}  Let $N$ be a set and let $\I=\set{\abar_i:i\in I}$ be any sequence of pairwise disjoint tuples in $N$.
	An {\em $\II$-partition of $N$} is any partition $N=\bigsqcup \set{A_i:i\in I}$ such that $\abar_i\subseteq A_i$ for each $i\in I$.
\end{definition}

We will ultimately use the following characterization of monadic NIP. Its crucial feature is that it is in terms of quantifier-free types.

\begin{proposition} \label{p:blobwidth}
	A theory $T$ is monadically NIP if and only if there is a cardinal $\lambda(T)$ such that for every $N \models T$ and every indiscernible sequence $\II = (\abar_i : i \in I)$ (equivalently, every indiscernible sequence of singletons $\II = (a_i : i \in I))$  in $N$ such that $I$ is a well-ordering with a maximum element, there is an $\II$-partition $(A_i : i \in I)$ of $N$ such that $\rtp_{qf}(N, A_{<i}) \leq \lambda(T)$  for every $i \in I$ (equivalently, $\rtp(N, A_{<i}) \leq \lambda(T)$  for every $i \in I$).
	
	Furthermore, for both $\rtp_{qf}$ and $\rtp$, we may take $\lambda(T) = 2^{|T|}$.
\end{proposition}
\begin{proof}
	$(\Ra)$ If $T$ is monadically NIP, then it has the f.s.\ dichotomy by \cite[Theorem 1.1]{MonNIP}. As in the proof of \cite[Lemma 4.4]{MonNIP}, we may find a model $M$ of size $|T|$ and an $\II$-partition $(A_i : i \in I)$ of $N$ that is also an $M$-f.s.\ sequence. Then every type realized in $A_{\geq i}$ over $A_{< i}$ is finitely satisfiable in $M$. The general bound on the number of global types finitely satisfiable in $M$ is $\beth_2(|T|)$, but \cite[Proposition 2.43]{Guide} improves this to $2^{|T|}$ under the assumption of NIP.

    $(\La)$ We prove the contrapositive. We wish to show that if $T$ is not monadically NIP, then we may find $\II$ (consisting of singletons) and $N$ such that  for every $\II$-partition $(A_i : i \in I)$ of $N$, there is an $i \in I$ with $\rtp_{qf}(N, A_{<_i})$ arbitrarily large. By \cite[Lemma 4.7]{MonNIP}, if this is true in some expansion of $T$ by finitely many unary predicates, then it is true in $T$ as well. Thus, up to replacing $T$ by some unary expansion with IP, we may assume $T$ has IP, rather than merely not being monadically NIP. By \cite[Lemma 4.6]{MonNIP},  if $\rtp(N, A_{<_i})$ can be made arbitrarily large, then so can $\rtp_{qf}(N, A_{<_i})$. 

    Fix a cardinal $\lambda$. From the above, we now assume $T$ has IP, and wish to find $\II$ (consisting of singletons) and $N$ such that  for every $\II$-partition $(A_i : i \in I)$ of $N$, there is an $i \in I$ with $\rtp(N, A_{<_i}) \geq \lambda$.
	
	By \cite{SimonNote} there is a formula (with parameters)  $\phi(x, y)$  on singletons witnessing IP. Let $N \models T$ contain an indiscernible sequence $\II = (a_i : i \leq \lambda)$ shattered by $\phi$, i.e. there is a set $Y = \set{b_s : s \in 2^\lambda} \subset N$ such that $N \models \phi(b_s, a_i) \iff i \in s$. Consider an $\II$-partition $(A_i : i \leq \lambda)$ of $N$. By pigeonhole and the fact that the cofinality of $2^\lambda$ is greater than $\lambda$, there is some $i^* \in I$ such that $A_{i^*}$ contains $Y' \subseteq Y$ of size $2^\lambda$.
	
	Partition $Y'$ according to the $\phi$-type of each element over $\II_{\leq a_{i^*}}$. If this partition has at least $\lambda$ classes then so does the partition using $\phi$-types over $\II_{< a_{i^*}}$, giving $\rtp(A_{i^*}, A_{<i^*}) \geq \lambda$. So suppose it has fewer classes. Then by pigeonhole there is some  $Y^* \subseteq Y'$ such that $|Y^*| = 2^\lambda$ and $\tp_\phi(b/\II_{\leq a_{i^*}})$ is constant among $b \in Y^*$. Thus $\tp_\phi(b/\II_{> a_{i^*}}) \neq \tp_\phi(b'/\II_{> a_{i^*}})$ for distinct $b, b' \in Y^*$. So there is $\II^* \subseteq \II_{> a_{i^*}}$ such that $|\II^*| = \lambda$ and $\tp_{\phi^{opp}}(a/Y^*) \neq \tp_{\phi^{opp}}(a'/Y^*)$ for distinct $a, a' \in \II^*$. Thus $\rtp(A_{>i^*}, A_{i^*}) \geq \lambda$, and so $\rtp(N, A_{<i^*+1}) \geq \lambda$, as desired.
\end{proof}

We remark that in most cases, we expect Proposition \ref{p:blobwidth} to be the simplest condition to use when showing that a given structure is monadically NIP, as it only requires working with quantifier-free types and indiscernible sequences of singletons. Furthermore, one may even try to work over sequences that are merely quantifier-free indiscernible (although then it may only be a sufficient condition).

We will now work toward proving the forward direction of the characterization above in the existentially closed setting, assuming the $\e$-f.s.\ dichotomy.

\begin{lemma} \label{l:e extend}
	Let $T$ be a universal theory. If $T$ has the  $\e$-f.s.\ dichotomy, then for every existentially closed $N \models T$ and $M \eprec N$, every partial $\e$-$M$-f.s.\ decomposition of $X \subset N$ can be extended to an  $\e$-$M$-f.s.\ decomposition of $X$.
\end{lemma}
\begin{proof}
	As in \cite[Lemma 3.2, Lemma 3.5]{MonNIP}, replacing complete types in $\C$ with existential types in $N$.
\end{proof}

\begin{definition}
	A sequence of tuples in a structure is {\em $\e$-indiscernible} if it is indiscernible with respect to all existential formulas.
	\end{definition}

The following Lemmas show that the basic results about indiscernibles and $M$-f.s.\ sequences carry over to $\e$-indiscernibles and $\e$-$M$-f.s.\ sequences.
We quote a   standard result for producing indiscernibles, based on the Erd\H{o}s-Rado theorem. See \cite[Proposition 1.6]{Cas} for a proof.

\begin{fact} \label{f:ER}
	Let $T$ be a complete theory, $\C$ a monster model for $T$, and $A \subset \C$. If $\kappa \geq |T| + |A|$ and $\lambda = \beth_{(2^{\kappa})^+}$, then for any sequence $(\abar_i : i < \lambda)$, there is an $A$-indiscernible sequence $(\bbar_i : i < \omega)$ such that for each $n < \omega$, there are $i_0, \dots, i_n < \lambda$ such that $\tp(\bbar_0, \dots, \bbar_n/A) = \tp(\abar_{i_0} \dots, \abar_{i_n}/A)$.
\end{fact}

\begin{lemma} \label{l:ind ext}
	Let $T$ be a universal theory. Let $M \models T$ and let $\II = (\abar_i : i \in I)$ be an indiscernible  sequence (for complete types) in $M$. Then there is some $N \supset  M$ such that $N$ is an existentially closed model of $T$ and such that $\II$ is $\e$-indiscernible in $N$.
\end{lemma}
\begin{proof}
	We first note it suffices to find some $|M|^+$-saturated $M^* \succ M$ containing an indiscernible sequence $\JJ$ of the same order type as $\II$ and with $\tp(\JJ) = \tp(\II)$, and to find an existentially closed $N^* \supset M^*$ such that $N^* \models T$ and $\JJ$ is $\e$-indiscernible in $N^*$. For then we may find an elementary embedding $f \colon M \to M^*$ with $f(\II) = \JJ$, which restricts to an isomorphism between $M$ and $f(M)$. Since  $N^*$ is a suitable extension of $f(M)$, we may find a suitable $N \supset M$.
	
	So we now aim to find $M^*, \JJ$, and $N^*$ as above. Let $\mu := 2^{|Th(M)|}$. Choose $M_1 \succ M$ containing $\II_1$ of order type $\beth_{\mu^+}$ and with $\tp(\II_1) = \tp(\II)$. Let $N_1 \supset M_1$ be an existentially closed model of $T$. We now expand $N_1$ by a unary predicate naming $M_1$, and consider the theory of the pair $(N_1, M_1)$.
	
	By Fact \ref{f:ER}, there is $(N_2, M_2) \succ (N_1, M_1)$ containing an indiscernible sequence $\JJ_0 = (\bbar_i : i \in \omega)$ such that for every $n \geq 1$, $\tp_{(N_2, M_2)}(\bbar_1, \dots, \bbar_n) = \tp_{(N_1, M_1)}(\abar_{i_1}, \dots, \abar_{i_n})$ for some $i_1, \dots, i_n \in \beth_{\mu^+}$. Then let $(N_3, M^*) \succ (N_2, M_2)$ be $|M|^+$-saturated and extending $\JJ_0$ to an indiscernible sequence $\JJ$ of the same order type as $\II$. Finally, let $N^* \supset N_3$ be an existentially closed model of $T$.
	
	By construction, $\JJ$ is indiscernible in $(N_3, M^*)$ and thus in both $N_3$ and $M^*$ separately. Furthermore, the $\e$-type of any tuple from $\JJ$ is maximal when considered in $N_3$, since this was true of any tuple from $\II$ considered in $N_1$ and thus any tuple from $\JJ_0$ considered in $N_2$ and $\JJ$ in $N_3$. Thus it remains true of any tuple from $\JJ$ considered in $N^*$. So since $\JJ$ was $\e$-indiscernible in $N_3$, it remains $\e$-indiscernible in $N^*$.
\end{proof}

\begin{lemma} \label{l:e ind to fs}
	Let $T$ be a universal theory. Let $N \models T$ be existentially closed, and contain a sequence $\I = (\abar_i : i \in I)$ that is $\e$-indiscernible over $\emptyset$. Then there are
	existentially closed $M^* \succ_\exists N$ and  $M \eprec M^*$ with $|M| = |T|$ such that $(\abar_i : i \in I)$ is both $\e$-indiscernible over $M$ and is an $\e$-$M$-f.s.\ sequence.
\end{lemma}
\begin{proof}  Fix $\kappa\ge |N|+|T|$.  
Since $N$ is existentially closed, $\etp(\abar_{i_1},\dots,\abar_{i_n})$ is maximal for every $i_1<\dots i_n$ from $I$.  
Choose an elementary extension $\C\succ N$ that is at least
$\kappa^+$-saturated.  
By Fact~\ref{f:ER} and stretching/thinning the index sets,  there is  an indiscernible sequence (with respect to all formulas) 
$\JJ=(\bbar_i:i\in I)$ in $\C$ such that, for each $n < \omega$ and every existential $\phi(\xbar_1,\dots,\xbar_n)$, we have that
$\phi(\bbar_{i_1},\dots,\bbar_{i_n})$ holds whenever 
$\phi(\abar_{i_1},\dots,\abar_{i_n})$ holds for some/every $i_1<\dots<i_n$.  
  That is,  $\etp^{\C}(\II)\subset\etp^{\C}(\JJ)$.  However, since $\etp^N(\II)$ is maximal, we have   $\etp^N(\II)=\etp^{\C}(\JJ)$.   
  As $\JJ$ is  indiscernible, it follows from \cite[Lemma 4.1]{Hanf} that there is some $M_J\preceq \C$ with $|M_J|=|T|$ such that $\JJ$ is indiscernible over $M_J$ and $\JJ$ is an $M_J$-f.s.\ sequence.

By Lemma~\ref{kappaecexists},  choose a $\kappa^+$-existentially closed model $M^*$ of $T$ containing $\II\JJ M_J$ of size $2^\kappa$.   Note that $f:\JJ\rightarrow\II$ sending $\bbar_i \mapsto \abar_i$ for every $i \in I$ is an $\e$-map.   By Lemma~\ref{auto} there is an existentially closed model $N^*\prec_\exists M^*$ of size $\kappa$ that has an automorphism $\sigma\in Aut(N^*)$ extending $f$.  Since $N^*$ is existentially closed,
our choice of $M_J$ gives that 
$\JJ$ is $\e$-indiscernible  over $M_J$ in $N^*$ and $\JJ$ is an $\e$-f.s.\ sequence over $M_J$.  
Put $M:=\sigma(M_J)$.
Then by the automorphism, $\II$ is $e$-indiscernible over $M$ and $\II$ is an $\e$-$M$-f.s.\ sequence.
\end{proof}

\begin{lemma} 	\label{l:e4.4} 
	Let $T$ be a universal theory with the $\e$-f.s.\ dichotomy. Then for every existentially closed $N \models T$ and every $\e$-indiscernible sequence $\II = (\abar_i : i \in I)$ in $N$ such that $I$ is a well-ordering with a maximum element, there is an $\II$-partition $(A_i : i \in I)$ of $N$ such that $\ertp(N, A_{<i}) \leq \beth_2(|T|)$ for every $i \in I$.
\end{lemma}
\begin{proof}
	By Lemma \ref{l:e ind to fs}, there is $N^* \succ_\exists N$ and $M \eprec N^*$ such that $|M| = |T|$ and $\II$ is an $\e$-$M$-f.s.\ sequence. By Lemma \ref{l:e extend}, we may extend this to an $\e$-$M$-f.s.\ decomposition $(A_i : i \in I)$ of $N$. By Lemma \ref{lem:basic}, $\etp(A_{\geq i}, A_{<i})$ does not $\e$-split over $M$. As there are at most $\beth_{2}(|T|)$ many global existential types that do not $\e$-split over $M$, the result follows.
	\end{proof}

\begin{lemma} \label{l:e-fs monNIP}
	Let $T$ be a universal theory with the $\e$-f.s.\ dichotomy. Then $T$ is monadically NIP.
\end{lemma}
\begin{proof}
	Let $M \models T$, and let $\II$ be an indiscernible sequence (for complete types) in $M$ that is well-ordered with a maximum element. By Lemma \ref{l:ind ext}, there is $N \supset M$ an existentially closed model of $T$ such that $\II$ remains $\e$-indiscernible in $N$. By Lemma \ref{l:e4.4}, $N$ admits an $\II$-partition $\set{B_i : i \in I}$ such that for every $i \in I$, $\rtp_{qf}(N, B_{<i}) \leq \ertp(N, B_{<i}) \leq \beth_2(|T|)$. Letting $A_i := B_i \cap M$, we have that $\set{A_i : i \in I}$ is an $\II$-partition of $M$ such that for every $i \in I$,  $\rtp_{qf}(M, A_{<i})\leq \beth_2(|T|)$. Thus $T$ is monadically NIP, by Proposition \ref{p:blobwidth}.
\end{proof}

Our second subgoal is to show that if $T$ does not have the $\e$-f.s.\ dichotomy, then this is witnessed by a particular configuration (a pre-coding configuration) defined by an existential formula.

The construction of this configuration in \cite[Proposition 3.11]{MonNIP} (although originally due to Shelah in \cite{Hanf}) is fundamentally based on the ability to extend a finitely satisfiable type to a larger set of parameters while maintaining finite satisfiability. As remarked in Subsection \ref{sec:ec}, this fails when working in existentially closed models, so we will have to move outside this setting to carry out the construction.

\begin{lemma}\label{l:fs lift}
	Let $T$ be a universal theory. Let $N \models T$ be existentially closed and let $M \eprec N$. Let $\abar, \bbar, c \in N$ witness a failure of the $\e$-f.s.\ dichotomy, i.e. $\etp(\bbar/M\abar)$ is finitely satisfiable in $M$, and there are formulas $\rho_1(\xbar, \ybar, z), \rho_2(\xbar, \ybar, z) \in \etp(\abar\bbar c/M)$ such that neither $\rho_1(\abar, \ybar, c)$ nor $\rho_2(\abar, \ybar, z)$ is satisfiable in $M$. Then there is  $N' \succ M$ such that $N \subset N'$, with $\bbar', c' \in N'$ such that $M, \abar, \bbar', c'$ witness the failure of the f.s.\ dichotomy, witnessed by the same formulas $\rho_1(\xbar, \ybar, z), \rho_2(\xbar, \ybar, z) \in \tp^\C(\abar\bbar' c'/M)$ such that neither $\rho_1(\abar, \ybar, c')$ nor $\rho_2(\abar, \ybar, z)$ is satisfiable in $M$.
\end{lemma}
\begin{proof}
	As $M$ and $N$ have the same universal theory, there is $N' \succ M$ such that $N \subset N'$ and $N'$ is $|M|^+$-saturated. Note that the existential types of elements in $N$ do not change when $N$ is embedded in $N'$, since they are already maximal.
	
	So, working in $N'$, we have $p(\ybar) := \etp(\bbar/M\abar)$ is unchanged (as are all the existential types of elements in $N$ over $M\abar$) and so it is still finitely satisfiable in $M$. Thus $p(\ybar)$ can be extended to a complete type $q(\ybar)$ that is finitely satisfiable in $M$ (e.g. by \cite[Fact 2.3(2)]{MonNIP}). Let $\bbar' \in N'$ be a realization of $q$. Index the elements of $M^{|\bbar|}$ as $(\mbar_i : i \in I)$. For each $i \in I$, as $\etp(\abar\mbar_i c)$ is maximal, there must be some existential formula $\psi_i(\xbar, \ybar, z)$ in the type that is an obstacle to $\rho_1$. For any finite $I_0 \subset I$, the formula $\exists z (\rho_1(\abar, \ybar, z) \wedge \rho_2(\abar, \ybar, z) \wedge \bigwedge_{i \in I_0} \psi_i(\abar, \mbar_i, z))$ is in $p \subset q$. So $r(z) = \set{\rho_1(\abar, \bbar', z) \wedge \rho_2(\abar, \bbar', z) \wedge \psi_i(\abar, \mbar_i, z)| i \in I}$ is finitely satisfiable in $N'$. Thus $r(z)$ is realized in $N'$ and we may choose any realization as $c'$.
\end{proof}

	\begin{definition} \label{def:pre-code}
		  A {\em pre-coding configuration} in a structure $M$ consists of a formula $\phi(\xbar,\ybar,z)$ with parameters and with $|\xbar| = |\ybar|$, a sequence $\I=(\dbar_i:i\in \Q)$ indiscernible over the parameters of $\phi$, and for all $s < t \in \Q$, a singleton $c_{s,t}$ such that the following holds.
	\begin{enumerate}
		\item  $M \models\phi(\dbar_s,\dbar_t,c_{s,t})$
		\item  $M \models \neg\phi(\dbar_s,\dbar_v,c_{s,t})$ for all $v>t$
		\item $M \models \neg\phi(\dbar_u,\dbar_t,c_{s,t})$ for all $u<s$
	\end{enumerate}

A theory {\em admits pre-coding} ({\em by an existential formula}) if some model contains a pre-coding configuration (with $\phi(\xbar, \ybar, z)$ existential).
\end{definition}

\begin{lemma} \label{l:e-pre}
	If a universal theory $T$ does not have the  $\e$-f.s.\ dichotomy, then it admits pre-coding by an existential formula.
\end{lemma}
\begin{proof}
	Suppose $T$ does not have the  $\e$-f.s.\ dichotomy, and let $N \models T$ be an existentially closed model such that $M \eprec N$ and $\abar\bbar c \subset N$ witness its failure, with existential formulas $\rho_1, \rho_2$ witnessing the required failures of finite satisfiability. Let $N' \succ M$ be as in Lemma \ref{l:fs lift}. Then $N'$ witnesses a failure of the f.s.\ dichotomy, with $\rho_1, \rho_2$ still witnessing the required failures of finite satisfiability, and so $N'$ admits pre-coding by \cite[Proposition 3.11]{MonNIP}. By the last sentence of the proof of the cited result, we may take the pre-coding formula to be $\rho_1 \wedge \rho_2$, which is existential.
\end{proof}

\begin{theorem} \label{thm:e pre-code}
For a universal theory $T$, the following are equivalent.
\begin{enumerate}
	\item $T$ is monadically NIP.
	\item $T$ does not admit pre-coding by an existential formula.
	\item $T$ has the $\e$-f.s.\ dichotomy
\end{enumerate}

In particular, if $T$ is not monadically NIP then it admits pre-coding by an existential formula.
\end{theorem}
\begin{proof}
	$(1) \Ra (2)$ If $T$ admits pre-coding (by any formula), then $T$ is not monadically NIP by \cite[Proposition 3.11, Theorem 1.1]{MonNIP}.
	
	$(2) \Ra (3)$ This is Lemma \ref{l:e-pre}.
	
	$(3) \Ra (1)$ This is Lemma \ref{l:e-fs monNIP}.
\end{proof}

We close this subsection by introducing a configuration that is related to pre-coding, but is easier to manipulate.
\begin{definition}  A {\em split configuration} in a structure $M$ consists of a quantifier-free formula $\psi(\xbar,\ybar,z,\wbar)$ (possibly with hidden parameters from $M$)
 together with disjoint, infinite index sets $(I,\le)$, $(J,\le)$, and infinite sets of tuples
$\{\abar_i:i\in I\}$, $\{\bbar_j:j\in J\}$, and $\{\cbar_{i,j}:(i,j)\in I\times J\}$ (with $|\cbar_{i,j}| = |z\wbar|$) from $M$ satisfying:  for all $i\in I$ and $j\in J$, putting $\phi(\xbar,\ybar,z);=\exists\wbar\psi$
and letting $\cbar^0_{i,j}$ denote the 0-th coordinate of $\cbar_{i,j}$,
\begin{itemize}
\item  $M\models\psi(\abar_i,\bbar_j,\cbar_{i,j})$
\item  For all $i^-<i$, $M\models \neg\phi(\abar_{i^-},\bbar_j,\cbar^0_{i,j})$
\item  For all $j^+>j$, $M\models\neg\phi(\abar_i,\bbar_{j^+},\cbar^0_{i,j})$
\end{itemize}
\end{definition}

Clearly, any pre-coding configuration by an existential formula gives rise to a split configuration by putting $I:=\Q^-$, $J:=\Q^+$, $\abar_i:=\dbar_i$, and $\bbar_j:=\dbar_j$.
 However,
a split configuration is more malleable, as e.g., we do not require $\lg(\abar_i)$ to equal $\lg(\bbar_j)$.
This freedom is used in the following lemma.

\begin{lemma}  \label{useRamsey}
  If $M,\psi,\{\abar_i:i\in I\}, \{\bbar_j:j\in J\},\{\cbar_{i,j}:(i,j)\in I\times J\}$ is a split configuration, then there is an elementary extension $M'\succeq M$
and a $\P$-indiscernible split configuration $\A$ inside $M'$ (of the same EM-type).  Moreover, among all quantifier-free $\psi$ yielding a split configuration, if we choose one where
$\lg(\wbar)$ is least possible, and then among those where $\lg(\xbar)+\lg(\ybar)$ is least possible, then the $\P$-indiscernible split configuration is a $\P$-indiscernible partition.
\end{lemma}

\begin{proof}   The first sentence follows by an analogue of the usual argument for producing indiscernible sequences using Ramsey's theorem and compactness, using that $Age(\P)$ has the  Ramsey property by Lemma~\ref{ageRamsey}. (See \cite[Theorem 3.12]{Scow} for a proof of a more general statement.)
For the second sentence, let $\psi$ satisfy the minimality conditions and suppose $\A$ is $\P$-indiscernible.  We argue that $\A$ is a $\P$-indiscernible partition.
First, if any $\abar_i$ had repeated values, then by $\P$-indiscernibility every $\abar_{i'}$ would have the same values repeated.  Thus, we could modify
$\psi$ to remove one of the free variables corresponding to the indices with equal values, contradicting our minimality assumption.  Thus, each $\abar_i$ is without repeats.  
 If  $\abar_i$ and $\abar_{i'}$ were not disjoint, then by $\P$-indiscernibility we would have some coordinate identically constant.
Thus, we could add an external  parameter for this common value and again get a smaller $\psi$.  Thus, $\{\abar_i:i\in I\}$ are pairwise disjoint 
with no repeated values.  Similar reasoning shows the same holds for $\{\bbar_j:j\in J\}$.  

Now we turn to $\cbar_{i,j}$.  First, if there were $i\neq k\in I$, $j\neq \ell\in J$ such that $\cbar_{i,j}$ and $\cbar_{k,\ell}$ were not disjoint, then by $\P$-indiscernibility there would be a constant value which we could remove.    Next, if there were $i\in I$ and $j\neq\ell\in J$ such that $\cbar_{i,j}$ and $\cbar_{i,\ell}$ were not disjoint, then by $\P$-indiscernibility
there would be a common value $c^*_i$ in some coordinate.  But then, by $\P$-indiscernibility there would be another value $c^*_{i'}$ for every other $i'\in I$.
Then we could take the same $\psi$, but rearrange the free variables making this coordinate a part of $\xbar$.  This would decrease $\lg(\wbar)$, contradicting our minimality assumption (even though $\lg(\xbar)$ is increased).  Similar reasoning shows that $\cbar_{i,j}$ and $\cbar_{i',j}$ are disjoint for all $i,i'\in I$ and $j\in J$.  Thus, 
$\{\cbar_{i,j}:(i,j)\in I\times J\}$ are pairwise disjoint.  
\end{proof}

\begin{corollary} \label{splitexists} Suppose $T$ is a universal $L$-theory that is not monadically NIP.
  Then in some expansion $L'\supseteq L$ by finitely many constants, there is an $L'$-structure  $M'\models T$ and a 0-definable, quantifier-free
  $\psi(\xbar,\ybar,z,\wbar)$ 
with a partitioned, $\P$-indiscernible split configuration
 \[\A:=\bigcup\{\{\abar_i:i\in I\},\{\bbar_j:j\in J\},\{\cbar_{i,j}:(i,j)\in I\times J\}\}\]
 via the mapping $f(i):=\abar_i$, $f(j):=\bbar_j$, $f(\gamma_{i,j}):=\cbar_{i,j}$. 
\end{corollary}

\begin{proof}  As $T$ is not monadically NIP, it admits a pre-coding configuration and hence a split configuration in some model.  Among all possible models of $T$ and all split configurations, fix one that first minimizes $\lg(\wbar)$ and among those, minimizes $\lg(\xbar)+\lg(\ybar)$ in the underlying $\psi(\xbar,\ybar,z,\wbar)$.
Expand $M$ by adding constants for the hidden parameters of $\psi$ and apply Lemma~\ref{useRamsey}.
\end{proof}

\subsection{Monadic stability}

The theorem in this section is proved for graph classes in \cite[Theorem 1.3]{rank}. Although the proof there generalizes immediately to other languages, we give a somewhat different proof.

\begin{definition}  Let $T$ be any (possibly incomplete) $\LL$-theory.   A partitioned $\LL$-formula $\phi(\xbar;\ybar)$ has the {\em order property with respect to $T$}
if there is some $M\models T$ such that,  for each $n\ge 1$, there are  $\{\abar_i:i\in[n]\}\subseteq M^{|\xbar|}$ and $\{\bbar_i:i\in[n]\}\subseteq M^{|\ybar|}$ such that
for all $i,j\in[n]$,
$$M\models\phi(\abar_i,\bbar_j)\quad \hbox{if and only if} \quad i< j$$
We say that an $\LL$-formula $\phi(\zbar)$ is {\em always stable} (with respect to $T$) if, for every partition $\zbar=\xbar\smallfrown\ybar$,
 $\phi(\xbar;\ybar)$ of $\phi(\zbar)$ does not have the order property with respect to $T$.
\end{definition}

Thus, an $\LL$-theory $T$ is unstable if and only if some $\phi(\zbar)$ is not always stable.  
It is well known that for any $\LL$-theory $T$, the set of always stable $\LL$-formulas $\theta(\zbar)$ is closed under boolean combinations.
Indeed, if $\phi(\xbar;\ybar):=\psi(\xbar;\ybar)\vee\theta(\xbar;\ybar)$ has the order property, then either $\psi(\xbar;\ybar)$ or $\theta(\xbar;\ybar)$ does as well by an application of Ramsey's theorem;  and if $\neg\phi(\xbar;\ybar)$ has the order property, then so does the `dual formula' $\phi(\ybar;\xbar)$.

\begin{lemma}  \label{qfbonus}
Suppose  $S$ and $T$ are $\LL$-theories satisfying
$S_\forall\subseteq T$.  If some quantifier-free $\LL$-formula
$\theta(\xbar,\ybar)$ has the order property with respect to $T$, then $\theta(\xbar,\ybar)$  also has the order property with respect to $S$.
\end{lemma}  

\begin{proof}  Choose any $M\models T$ witnessing that $\theta(\xbar,\ybar)$ has the order property with respect to $T$.
By Fact~\ref{f:univ sub}, there is a model $N\supseteq M$ of $S$.  Choose any $n\ge 1$ and  $\{\abar_i:i\in[n]\}\subseteq M^{|\xbar|}$, $\{\bbar_i:i\in[n]\}\subseteq M^{|\ybar|}$
witnessing that $\theta(\xbar,\ybar)$ has the order property in $M$.  Since $\theta(\xbar,\ybar)$ is quantifier free, the same tuples witness that $\theta(\xbar,\ybar)$ has the order property in $N$.
\end{proof}

\begin{lemma} \label{claimtransfer}  Suppose $T$ is a complete, monadically NIP $\LL$-theory for which $\psi(\xbar,\ybar)
		 := \exists z \phi(\xbar; \ybar; z)$ has the order property with respect to $T$.  Then either $\phi(\xbar z; \ybar)$ or $\phi(\xbar; \ybar z)$  also has the order property with respect to $T$.
\end{lemma}

\begin{proof}
		We work in a large saturated model $\C$ of $T$. Since $\psi(\xbar;\ybar)$ has the order property, there is an indiscernible sequence $\II := (\abar_i\bbar_i : i \in \Z)$ such that $\C \models \psi(\abar_i, \bbar_j) \iff i < j$. 
	
First, assume that for every $n\ge 1$, $\C\models\exists z\bigwedge_{j=1}^n \phi(\abar_0,\bbar_j,z)$.  Under this assumption, the indiscernibility of $\II$ and compactness imply that for every $i\in \Z$,
$$p_i(z):=\{\phi(\abar_i,\bbar_j,z):i<j\}\cup\{\neg\phi(\abar_i,\bbar_j,z):j\le i\}$$
is consistent.  As $\C$ is sufficiently saturated, choosing $c_i$ to realize $p_i$, the sequence $(\abar_i\bbar_ic_i)$ witnesses that $\phi(\xbar z;\ybar)$ has the order property.
Similarly, if for every $n\ge 1$, $\C\models\exists z\bigwedge_{j=1}^n \phi(\abar_{-j},\bbar_0,z)$, then we can find a sequence $(d_j)$ so that $(\abar_i\bbar_id_i)$ witnesses that
$\phi(\xbar;\ybar z)$ has the order property.  

Finally, assume neither of these properties hold, and we will contradict $T$ being monadically NIP.  Choose any $s<t$ from $\Z$ and choose $c\in\C$ such that $\C\models\phi(\abar_s,\bbar_t,c)$.
From the negation of the first case above and by indiscernibility, there is a finite bound (independent of $s$) on the number of elements $\bbar_j$ such that $\C \models \phi(\abar_s, \bbar_j, c)$. Similarly, from the negation of the second case, there is a finite bound (independent of $t$) on the number of elements $\abar_i$ such that $\C \models \phi(\abar_i, \bbar_t, c)$. Thus there is a finite $F \subseteq \Z$ such that, letting $J := (\Z \bs F) \cup \set{s,t}$, we have $\C\models\neg\phi(\abar_s,\bbar_j,c)$ for every $j \in J\bs\set{t}$ and $\C\models\neg\phi(\abar_i,\bbar_t,c)$ for every $i\in J\bs\set{s}$. 
That is, $c$ witnesses the partial type $\Gamma_{s,t}(z)$ over $\bigcup\{\abar_j\bbar_j:j\in J\}$asserting 
\[\hbox{for all $(i,j)\in J^2$},\quad \C\models\phi(\abar_i,\bbar_j,z) \quad  \hbox{if and only if} \quad  (i,j)=(s,t)\]
By the indiscernibility of $\II$ and the saturation of $\C$, it follows that for every $i^*<j^*$ there is $c_{i^*,j^*}$ such that, for all pairs $(i,j)\in \Z^2$,
$\C\models \phi(\abar_i,\bbar_j,c_{i^*,j^*})$ if and only if $(i,j)=(i^*,j^*)$.  This is a pre-coding configuration, contradicting $T$ being monadically NIP.
\end{proof}

\begin{theorem} \label{thm:qf op} 
If $T$ is a monadically NIP $\LL$-theory that is not monadically stable, then some atomic $\LL$-formula is not always stable with respect to every $\LL$-theory
$S$ with $S_\forall\subseteq T$.  
\end{theorem}

\begin{proof}  As $T$ is not monadically stable, choose $M\models T$ and a monadic  expansion $M^*$ of $M$ in a language $\LL^*=\LL\cup\{U_1,\dots,U_n\}$
for which $T^*=Th(M^*)$ is unstable.   Note that if every atomic $\LL^*$-formula were always stable, then by Lemma~\ref{claimtransfer} (applied to $T^*$) 
and the the closure of always stability under boolean combinations, an induction on  quantifier complexity would imply that every $\LL^*$-formula is always stable, contradicting $T^*$ unstable.
Thus, 
some atomic $\LL^*$-formula $\alpha^*$ is not always stable.
Clearly, $\alpha^*$ cannot be unary, hence $\alpha^*$ is an atomic $\LL$-formula.  Say $\alpha^*(\xbar;\ybar)$ has the order property with respect to  $T^*$.
Then, as $T\subseteq T^*$,  Lemma~\ref{qfbonus} implies $\alpha^*(\xbar;\ybar)$ also has the order property with respect to $S$ for 
every $\LL$-theory $S$ with  $S_\forall\subseteq T$.
\end{proof}

\section{Collapse of dividing lines} \label{sec:app}

We now apply the results of the previous section to hereditary classes of finite structures. For hereditary classes, it is common to use more finitary definitions of (monadic) NIP/stability than Definition \ref{def:prop}. We begin by showing these are equivalent for NIP/stability. The equivalence for monadic NIP/stability will use the main result of this section, and so appear at its end in Proposition \ref{prop:findef}.

	\begin{definition}
	For a formula $\phi(\xbar;\ybar)$ with its free variables partitioned and a bipartite graph $G=(I,J;E)$, we say a structure $M$ {\em encodes $G$ via $\phi$} if
	there are sets $A=\set{\abar_i | i \in I} \subseteq M^{|\xbar|}, B=\set{\bbar_j | j \in J} \subseteq M^{|\ybar|}$ such that 
	$M\models\phi(\abar_i,\bbar_j)\Leftrightarrow G\models E(i,j)$.  
	
	Given a class $\CC$ of structures, {\em $\phi$ encodes $G$ in $\CC$} if there is some $M_G \in \CC$ encoding $G$ via $\phi$.
\end{definition}

\begin{notation} \label{not:th(c)}
	Given a class $\CC$ of structures, we use $Th(\CC) := \bigcap_{M \in \CC} Th(M)$ to denote the common theory of structures in the class.
\end{notation}

\begin{lemma} \label{lem:th(C)}
	Let $\CC$ be a class of structures. Then $Th(\CC)$ is NIP if and only if no formula encodes every finite bipartite graph in $\CC$, and $Th(\CC)$ is stable if and only if no formula encodes every finite half-graph in $\CC$.
\end{lemma}
\begin{proof}
	$(\Ra)$ Immediate by compactness.
	
	$(\La)$ Suppose that for every partitioned formula $\phi(\xbar;\ybar)$ there is some bipartite (half) graph $G_{\phi(\xbar;\ybar)}$ that $\phi(\xbar;\ybar)$ does not encode in $\CC$. This may be expressed by a sentence for each $\phi(\xbar;\ybar)$, each of which will be in $Th(\CC)$.
	\end{proof}

Lemma \ref{lem:th(C)} suggests considering $Th(\CC)$ when given a hereditary class $\CC$. It would also be natural to view $\CC$ as corresponding to a universal theory and to consider $Th(\CC)_{\forall}$ instead, which is often an easier theory to work with as we may pass to arbitrary substructures of infinite models. Since $Th(\CC)_{\forall} \subseteq Th(\CC)$, it may have more models; for example, if $\CC$ is the class of finite linear orders, then $(\Q, <)$ is a model of $Th(\CC)_{\forall}$ but not of $Th(\CC)$, since all models of $Th(\CC)$ are discrete. Nevertheless, we shall see in Theorem \ref{thm:coll} that for deciding whether a hereditary class is (monadically) stable/NIP, it does not matter whether we consider $Th(\CC)$ or $Th(\CC)_{\forall}$.

The main results of this section will follow quickly from the next lemma showing the collapse between NIP and monadic NIP in hereditary classes.

\begin{lemma} \label{l:collNIP}
	Let $\CC$ be a hereditary class in a relational language. If $Th(\CC)_{\forall}$ is not monadically NIP, then $Th(\CC)$ is not NIP.
\end{lemma}
\begin{proof}
	Suppose $Th(\CC)_{\forall}$ is not monadically NIP, and so by Corollary~\ref{splitexists} there is an expansion $L'\supseteq L$ by finitely many constants,
	an $M'\models Th(\CC)_{\forall}$, a quantifier-free $L'$-formula $\psi(\xbar,\ybar,z,\wbar)$ and a partitioned, $\P$-indiscernible split configuration indexed by
	$\{\abar_i:i\in \Q^-\}$, $\{\bbar_j:j\in\Q^+\}$ and $\{\cbar_{i,j}: i \in \Q^-, j \in \Q^+\}$.    Let $\phi(\xbar,\ybar,z)$ denote $\exists\wbar\psi(\xbar,\ybar,z,\wbar)$.
	For each integer $i\in \Z^-$, put $\abar^*_i:=\abar_{i-1/4}\abar_i$ and, for each $j\in \Z^+$, put $\bbar^*_j:=\bbar_j\bbar_{j+1/4}$.  
	Also, put $\xbar^*:=\xbar'\xbar$ and $\ybar^*:=\ybar\ybar'$ and put
	$$\theta(\xbar^*,\ybar^*):=\exists z[\phi(\xbar,\ybar,z)\wedge\neg\phi(\xbar',\ybar,z)\wedge\neg\phi(\xbar,\ybar',z)]$$
	
	We will show that $Th(\CC)$ is not NIP by showing that for every finite, bipartite graph $G=(S,T;E)$ there is a finite substructure $M_G'\subseteq M'$ encoding $G$ via 
	the $L'$-formula $\theta(\xbar^*,\ybar^*)$, which suffices by Lemma \ref{lem:th(C)}.

Let $m:= |\cbar_{i,j}|+1$. For every integer $j\in\Z$, let $P(j):=\{j+{\frac{k}{ 4m}}:-2m<k<2m\}$, which is a finite set of rationals in the interval $(j-1/2,j+1/2)$.
	For disjoint finite sets of integers $S\subseteq \Z^-$,  $T\subseteq \Z^+$, let $I_{ST}:=\bigcup\{P(j):j\in S\cup T\}$.  Clearly, if $j<j'$ are integers, then $P(j)\ll P(j')$.
	The salient feature of $I_{ST}$ is that if $j\in S\cup T$ and $Z\subseteq P(j)$ is of size $\le m$ with $j \in Z$, then there are order-preserving functions $g,h\colon Z\rightarrow P(j)$
	such that $g(j)=j+1/4$ and $h(j)=j-1/4$. 
	
	Let  $M'_\emptyset$ be the finite substructure of $M'$ with universe
	$$\pbar\cup\{\abar_s:s\in I_{ST}, s<0\}\cup\{\bbar_t:t\in I_{ST}, t>0\}$$
	where $\pbar$ is the finite tuple of $M'$ named by constants
	and let
	$M'_{ST}$ be the finite substructure of $M'$ with universe 
	$$M'_\emptyset\cup \{\cbar_{s,t}:s\in\bigcup_{i\in S} P(i),t\in\bigcup_{j\in T} P(j)\}$$
	Let $j  \in T$, $Z \subset P(j)$, and $g \colon Z \rightarrow P(j)$ be as above, i.e. $g$ is order-preserving and $g(j) = j+1/4$. Let $Z' \subset M'_{ST}$ consist of all points with (all) indices in $Z \cup (I_{ST}\bs P(j))$. Then we define $g'\colon Z' \to M'_{ST}$ as follows.
	\begin{itemize}
	\item  If $v\in\bbar_z$ for $z \in Z$, then $g'(v)$ is the associated element of $\bbar_{g(z)}$.
	\item  If $v\in\cbar_{r,z}$ for $z \in Z$, then $g'(v)$ is the associated element of $\cbar_{r,g(z)}$.
	\item Otherwise, $g'(v) = v$ (noting that $g(z)=z$).
	\end{itemize}
Coupled with the $\P$-indiscernibility,  this yields that for any $m$-tuple
$\vbar$, $\tp^{M'}(\vbar)=\tp^{M'}(g'(\vbar))$ for any of the functions $g'\colon Z' \to M'_{ST}$ described above.

Now, for any bipartite graph $G=(S,T;E)$, let $M'_G\subseteq M'_{ST}$ be the substructure with universe 
$$M'_\emptyset\cup\bigcup_{(i,j)\in E}\{\cbar_{r,s}:r\in P(i),s\in P(j)\}$$

Note that for any $g'$ as above, its restriction to $M'_G$ has its image contained within $M'_G$. Clearly, $M'_G$ need not be $\P$-indiscernible for all formulas, but it will
be $\P$-indiscernible for quantifier-free formulas, and so for any $m$-tuple $\vbar$ in $M'_G$,  $\qftp^{M'_G}(\vbar) = \qftp^{M'_G}(g'(\vbar))$.

	\begin{claim*}
		$M'_G \models \theta(\abar_i^*, \bbar_j^*)$ if and only if $(i,j)$ is an edge in $G$.
	\end{claim*}
	\begin{claimproof}
		If $(i,j)$ is an edge in $G$, then $\cbar_{i,j} \subset M'_G$, so $\cbar^0_{i,j}$ witnesses the outermost existential in $\theta$, with the rest of $\cbar_{i,j}$ witnessing the existentials in $\phi(\abar_i, \bbar_j, c_{i,j})$. Since $M' \models \neg \phi(\abar_{i - 1/4}, \bbar_j, \cbar^0_{i,j}) \wedge \neg \phi(\abar_{i}, \bbar_{j+1/4}, \cbar^0_{i,j})$ and $\neg \phi$ is universal, this is true in $M'_G$ as well.
		
		Now suppose $(i,j)$ is not an edge in $G$. Suppose there are $e^0 \in M'_G$ and $\ebar \subset M'_G$ such that $M'_G \models \psi(\abar_i, \bbar_j, e^0\ebar)$. 
		To show $M'_G\models\neg\theta(\abar_i,\bbar_j)$ it suffices to find some  $\ebar' \subset M'_G$ such that either $M'_G \models \psi(\abar_{i-1/4}, \bbar_j, e^0\ebar')$ or $M'_G \models \psi(\abar_i, \bbar_{j+1/4}, e^0\ebar')$.   Note that $e^0\in M'_G$ has at most two indices, and since $(i,j)\not\in E$, they cannot be in both $P(i)$ and $P(j)$.
		For definiteness, suppose no index of $e^0$ is in $P(j)$.  Let $Z\subseteq P(j)$ be the indices in $\bbar_j\ebar$ contained in $P(j)$, and note $|Z| \leq m$. Let $Z'$ and $g' \colon Z' \to M'_{ST}$ be as above; so $g'$ is order-preserving on the indices of points, sends points with index $j$ to $j + 1/4$, and fixes points whose indices are outside $P(j)$. Thus $g'(\abar_i\bbar_je^0\ebar)=(\abar_i\bbar_{j+1/4}e^0g'(\ebar))$. By the paragraph before this Claim, $g'(\ebar) \subset M'_G$ and $M'_G\models\psi(\abar_i,\bbar_{j+1/4},e^0g'(\ebar))$, which suffices.
	\end{claimproof}
\end{proof}

\begin{theorem} \label{thm:coll}
	Let $\P \in \set{\textrm{stable, NIP}}$. Let $\CC$ be a hereditary class of relational structures. Then the following are equivalent.
	\begin{enumerate}
		\item $Th(\CC)_{\forall}$ is monadically $\P$.
		\item $Th(\CC)$ is monadically $\P$.
		\item $Th(\CC)_{\forall}$ is $\P$.
		\item $Th(\CC)$ is $\P$.
	\end{enumerate}
\end{theorem}
\begin{proof}
	The equivalence of $(1)$ and $(2)$ is by Proposition \ref{p:mon univ}, and clearly $(1) \Ra (3) \Ra (4)$. So it only remains to show $(4) \Ra (1)$. We will show the contrapositive in each case.
	
	Case A: $\P =$ NIP. This is Lemma \ref{l:collNIP}.
	
	Case B: $\P =$ stable. Suppose $Th(\CC)_{\forall}$ is not monadically stable, and so neither is $Th(\CC)$, by Proposition \ref{p:mon univ}. By Theorem \ref{thm:qf op}, either $Th(\CC)$ is not monadically NIP and thus not even NIP by Case A, or it is monadically NIP but unstable. In either case, we are finished.
\end{proof}

\begin{corollary}  \label{stengthen2.5}
Let $T$ be a (possibly incomplete) theory in a relational language, and let $\PP \in \set{\textrm{stable, NIP}}$. Then $T$ is monadically $\PP$ if and only if $T_{\forall}$ is $\PP$.
\end{corollary}
\begin{proof}
	From Proposition \ref{p:mon univ}, it suffices to show that if $T_\forall$ is $\PP$ then $T_\forall$ is monadically $\PP$. Let $\CC$ be the hereditary class of finite models of $T_\forall$, so $Th(\CC)_\forall = T_\forall$. The result now follows from Theorem \ref{thm:coll}.
\end{proof}

The following definition is analogous to monotone graph classes,  which are closed under (not-necessarily-induced) subgraph, or equivalently under the removal of vertices and edges.

\begin{definition} \label{def:mono}
	A {\em monotone class} $\CC$ is a hereditary class with the additional property that for every $M \in \CC$, every structure obtained from $M$ by removing instances of atomic relations (other than equality and non-equality) is still in $\CC$.
\end{definition}

Note that this definition is not exactly a generalization of monotone graph classes, since for graphs the edge relation must be symmetric, so if we remove the relation $E(a,b)$ we must also remove $E(b,a)$. Placing such additional symmetry constraints on the relations of $\CC$ would not change the proof of the next theorem.

We now generalize part of the main result of \cite{AdAd} from graphs to relational structures, answering part of \cite[Problem 5.1]{Spars}.

\begin{theorem} \label{thm:coll mon}
	Let $\CC$ be a monotone class of relational structures. Then $Th(\CC)$ is NIP if and only if $Th(\CC)$ is monadically stable.
\end{theorem}
\begin{proof}
	The backward direction is immediate. For the forward direction, by Theorem \ref{thm:coll} it suffices to show that if $Th(\CC)_{\forall}$ is monadically NIP then it is stable. So suppose it is monadically NIP but unstable, and we will create a contradiction by showing it is not NIP. By  Theorem \ref{thm:qf op}, there is an atomic formula $\phi(\xbar; \ybar)$ with the order property in $Th(\CC)_{\forall}$. Let $(\abar_i : i \in \Z)$, $(\bbar_j : j \in \Z)$ be in $M \models Th(\CC)_{\forall}$ such that $M \models \phi(\abar_i; \bbar_j) \iff i < j$.
	So $\phi$ defines a complete bipartite graph on $\set{\abar_i | i \in \Z^-} \times \set{\bbar_j | j \in \Z^+}$. As $\CC$ is monotone and $\phi$ is atomic, we may produce a model of $Th(\CC)_{\forall}$ by removing instances of $\phi$ so the remaining instances define the random bipartite graph. Thus $Th(\CC)_{\forall}$ has IP.
\end{proof}

Finally, we show that our definition of monadic NIP/stability in Definition \ref{def:prop} corresponds to the definition usually used in the finite combinatorics literature for a class of structures. That definition is in terms of transductions, and amounts to $(3)$ in the following proposition.

\begin{proposition} \label{prop:findef}
   Let $\CC$ be a class of structures. The following are equivalent.
    \begin{enumerate}
        \item $\CC$ is monadically NIP (resp. monadically stable).
        \item Every class $\CC^+$ obtained by expanding the structures in $\CC$ by unary predicates is NIP (resp. stable).
        \item No class $\CC^+$ obtained by expanding the structures in $\CC$ by unary predicates has a formula on singletons encoding all finite bipartite graphs (resp. all half graphs).
    \end{enumerate}
\end{proposition}
\begin{proof}
Since all of these properties are preserved by closing under substructure, we may assume $\CC$ is hereditary. We only consider the case of NIP, since the argument for stability is essentially identical.

	$(1) \Ra (2)$ Suppose $Th(\CC)$ is monadically NIP and let $\CC^+$ be as described. Let $M^+ \models Th(\CC^+)$, and let $M$ be its reduct to the original language. Then $M \models Th(\CC)$, and so is monadically NIP, and thus $M^+$ is NIP.
	
	$(2) \Ra (1)$ Suppose $Th(\CC)$ is not monadically NIP. Then there is $M \models Th(\CC)$ admitting a unary expansion $M^+$ that has IP. Let $\CC^+ := Age(M^+)$, so $\CC^+$ is contained in a unary expansion of $\CC$. Then $M^+ \models Th(\CC^+)_{\forall}$ has IP, and thus $Th(\CC^+)$ has IP by Theorem \ref{thm:coll}.

    $(2) \Ra (3)$ This is immediate from Lemma \ref{lem:th(C)}.

    $(3) \Ra (2)$ Suppose $(2)$ fails, and let $\CC^+$ be a unary expansion of $\CC$ with IP. Thus there is some theory $T^+$ a completion of $Th(\CC^+)$ with IP. By \cite{SimonNote}, this is witnessed by a formula $\phi(x;y)$ on singletons with parameters. Let $T^{++}$ be the theory obtained by expanding by constants for these parameters. Then there is some expansion $\CC^{++}$ of $\CC^+$ by constants so that $T^{++}$ is a completion of $Th(\CC^{++})$. By (the proof of) Lemma \ref{lem:th(C)}, $\phi$ encodes all finite bipartite graphs in $\CC^{++}$. Finally, the constants used for the parameters of $\phi$ can be replaced with unary predicates, and the formula $\phi$ adjusted accordingly, yielding a failure of $(3)$.
\end{proof}

\section{Many models} \label{sec:growth}

We close with an application to the finite combinatorics of hereditary classes. We show that if a class $\CC$ of finite relational structures is not monadically NIP then it has superexponential growth rate in the following sense, removing the hypothesis of quantifier elimination from a result in \cite{MonNIP}.

\begin{definition}
	Let $\CC$ be a hereditary class. The {\em (unlabeled) growth rate of $\CC$} is the function $f_{\CC}(n)$ counting the number of isomorphism types in $\CC$ with $n$ elements.
\end{definition}

The class of all graphs with degree at most five is monadically NIP (in fact mutually algebraic, and thus monadically NFCP \cite{MA}) and has labeled growth rate $\Omega(n^{5n/2})$\cite[Formula (6.6)]{Noy}, so its unlabeled growth rate is $\Omega(n^{3n/2})$. Thus the converse to our theorem does not hold, and monadic NIP does not separate classes according to their growth rates, since this example's unlabeled growth rate is faster than that of the class of permutations viewed as structures with two linear orders, which is not monadically NIP. However, under the additional assumption that $\CC$ is the class of finite substructures of an $\omega$-categorical structure, we have conjectured that monadic NIP implies $f_\CC(n)$ is at most exponential \cite[Conjecture 1]{MonNIP}.

Fast growth rate will be a quick consequence of the following non-structure result, which we isolate for possible further applications. Like Lemma \ref{l:collNIP}, we will be encoding bipartite graphs in finite structures, but we allow ourselves to expand the language by unary predicates. Using this, the sets of tuples on which we encode bipartite graphs will be made definable, which then allows us to define the graphs on singletons. This much was already shown in \cite[Theorem 8.1.8]{BS}, but our approach allows us to additionally bound the size of the structure we are using to define a given graph.

\begin{proposition} \label{prop:transd}
	Let $\CC$ be a hereditary class in a relational language $\LL$. If $Th(\CC)$ is not monadically NIP, then there is an expansion $\LL^* \supset \LL$ by finitely many unary predicates and a corresponding expansion $\CC^*$ of $\CC$,  and $\LL^*$-formulas $U^*(x)$, $V^*(x)$, and $E^*(x, y)$ on singletons such that for every finite bipartite graph $G = (U,V; E)$, there is $M^*_G \in \CC^*$ such that $G \cong (U^*(M^*_G), V^*(M^*_G); E^*(M^*_G))$ and $|M^*_G| = O(|U|+ |V| + |E|)$.
\end{proposition}

\begin{proof}  As $Th(\CC)$ is not monadically NIP, by Corollary \ref{splitexists} we can find a saturated model $\C\models Th(\CC)$ with a partitioned, $\P$-indiscernible $\A\subseteq\C$
for which there is a quantifier-free $\psi(\xbar,\ybar,z,\wbar)$ such that, letting $\phi(\xbar,\ybar,z):=\exists\wbar\psi$ and letting $\cbar^0_{i,j}$ denote the $0^{th}$ coordinate of $\cbar_{i,j}$
we have, for all $i\in I$, $j\in J$,
\begin{itemize}
\item  $\C\models\psi(\abar_i,\bbar_j,\cbar_{i,j})$;
\item  $\C\models\neg\phi(\abar_{i'},\bbar_j,\cbar^0_{i,j})$ for all $i'<i$ from $I$; and
\item  $\C\models\neg\phi(\abar_i,\bbar_{j'},\cbar^0_{i,j})$ for all $j'>j$ from $J$.
\end{itemize}

We now `duplicate' $\A$ according to Remark~\ref{duplication} twice, once for $I$ and once for $J$.
Let $\xbar^*:=\xbar'\xbar$ and $\ybar^*:=\ybar\ybar'$, and (defining $I_0, I_1, J_0, j_1$  $i^-$ as in Remark \ref{duplication}, as well as $i^-$ and $j^+$) let $\abar_i^*:=\abar_{i^-}\abar_i$ for each $i\in I_1$
and let $\bbar_j^*:=\bbar_j\bbar_{j^+}$ for each $j\in J_0$.  
Also, put $$\chi(\xbar^*,\ybar^*,z):=\phi(\xbar,\ybar,z)\wedge\neg
\phi(\xbar',\ybar,z)\wedge\neg\phi(\xbar,\ybar',z)$$
We obtain, for each $i,k\in I_1$ and each $j,\ell\in J_0$,
$$\C\models \chi(\abar_i^*,\bbar_j^*,\cbar^0_{k,\ell})\quad \hbox{if and only if} \quad (i,j)=(k,\ell)$$	

Next, we look at all specializations $\chi'(\xbar',\ybar',z)$ formed by replacing some of the free variables in $\xbar^*$ by constants representing elements of $\bigcup\{\abar_i^*:i\in I_1\}$
and replacing some of the free variables in $\ybar^*$ by constants representing elements of $\bigcup\{\bbar_j^*:j\in J_0\}$.
Call such a specialization {\em allowable} if there are convex subsets $I'\subseteq I_1$, $J'\subseteq J_0$ such that all of the parameters added are from $\bigcup\{\abar_i^*:i\in I_1\setminus I'\}\cup\bigcup\{\bbar_j^*:j\in J_0\setminus J'\}$ and $\C$ still satisfies
$$\C\models \chi'(\abar_i',\bbar_j',\cbar^0_{k,\ell})\quad \hbox{if and only if} \quad (i,j)=(k,\ell)$$
for all $i,k \in I'$, $j,\ell\in J'$, where $\abar'_i$ is the restriction of $\abar_i^*$ to the free variables $\xbar'$ and dually for $\bbar_j'$.

Clearly, $\chi$ itself is allowable, taking $I'=I_1$ and $J'=J_0$.  Among all such allowable specializations, choose one that minimizes $\lg(\xbar')+\lg(\ybar')$ and
add these  new constant symbols to the language.

After reindexing,  replacing $I_1$ by $I'$, $J_0$ by $J'$, $\abar^*_i$ by $\abar'_i$, and $\bbar^*_j$ by $\bbar'_j$, we have (by Remark~\ref{P0} applied to $I',J'$) a partitioned, 
$\P$-indiscernible (in this larger language) $\A'$
and $\chi'(\xbar',\ybar',z)$, which is a boolean combination of specializations of $\phi$
such that $$\C\models \chi'(\abar_i',\bbar_j',\cbar^0_{k,\ell})\quad \hbox{if and only if} \quad (i,j)=(k,\ell)$$
for all $i,k\in I'$ and $j,\ell\in J'$.  The additional minimality property we have gained on $\chi'$ will be used in the proof of Claim~\ref{second}. Let $m_a = \lg(\abar_i')$, $m_b = \lg(\bbar_j')$, and $m_c = \lg(\cbar_{i,j})$.

After doing this minimization, 
let $\C^+$ denote the expansion of $\C$ by the unary predicates defining the strips of $\A'$; we use $A^j$ to define the $j^{th}$ coordinate strip of $(\abar_i : i \in I)$ for $0 \leq j \leq m_a-1$, and similarly $B^j$ and $C^j$.  By Remark~\ref{addstrips}, $\A'$ remains $\P$-indiscernible 
in $\C^+$ with respect to all $L^+$-formulas. Now, to further simplify the notation, remove the primes from all of the items discussed above.

Let $\chi^+(\xbar,\ybar,z)$ be the same formula as $\chi$, except that we add a conjunct stating that $C^0(z)$ holds, and in each instance of $\phi$ or its negation, we replace $\exists \wbar$ by
$\exists w_1\in C^1\exists w_2\in C^2\dots\exists w_{m_c-1}\in C^{m_c-1}$.  Note that we still have 
$$\C^+\models \chi^+(\abar_i,\bbar_j,\cbar^0_{k,\ell})\quad \hbox{if and only if} \quad (i,j)=(k,\ell)$$
for all $i,k \in I$, $j,\ell\in J$.

We now define various subsets of the index sets $(I,\le)$ and $(J,\le)$, recalling that both of these are isomorphic to $(\Q,\le)$.
Let $D_I$ be a discrete subset of $I$ of order type $\omega$, and let $i^*$ denote the least element of $D_I$.
For each $i\in D_I$, choose  `neighbors' $i^-,i^+$ such that $i^-<i<i^+$ and, letting $Nb(i)=\{i^-,i,i^+\}$,  such that
$Nb(i)\ll Nb(i')$ whenever $i<i'$ in $D_I$.

Put $3D_I:=\bigcup\{Nb(i):i\in D_I\}$.  For each $i\in 3D_I$, let its `cilia' $Cil(i)$ consist of $2m_c+1$ points, centered at $i$, 
such that $Cil(i)\ll Cil(i')$ whenever $i<i'$ in $3D_I$.

Similarly, let $D_J\subseteq J$ be discrete of order type $\Z^{\leq 0}$ with largest element $j^*$, and define
$N(j)$, $3D_J$, and $Cil(j)$ analogously.  

Let $\C^{++}$ be the expansion of $\C^+$ formed by  naming each element of $\abar_{i^*}$ and $\bbar_{j^*}$ by constant symbols and adding the following five unary predicates:
\begin{itemize}
\item  $E_I:=\{\cbar^0_{i,j^*}:i\in D_I\}$;
\item  $F_I:=\{\cbar^0_{i,j}:i\in 3D_I,j\in Nb(j^*)\}$;
\item  $E_J:=\{\cbar^0_{i^*,j}:j\in D_J\}$; 
\item  $F_J:=\{\cbar^0_{i,j}:i\in Nb(i^*),j\in 3D_J\}$; and
\item  $N:=\bigcup\{\abar_i:i\in D_I\}\cup\bigcup\{\bbar_j:j\in D_J\}\cup\bigcup\{\cbar_{i,j}:i\in Cil(3D_I),j\in Nb(j^*)\}\cup \bigcup \{\cbar_{i,j}:i\in Nb(i^*), j\in Cil(3D_J)\}$
\end{itemize}

We also define two $L^{++}$-formulas 
$$\alpha(\xbar):=\bigwedge_{\ell<m_a} A^\ell(x_\ell)\wedge \exists z (E_I(z)\wedge \chi^+(\xbar,\bbar_{j^*},z))
\wedge \forall z[(\chi^+(\xbar,\bbar_{j^*},z)\wedge F_I(z))\rightarrow E_I(z)]$$

$$\beta(\ybar):=\bigwedge_{\ell<m_b} B^\ell(y_\ell)\wedge \exists z (E_J(z)\wedge \chi^+(\abar_{i^*},\ybar,z))
\wedge \forall z[(\chi^+(\abar_{i^*},\ybar,z)\wedge F_J(z))\rightarrow E_J(z)]$$

Now, given any finite sets $S\subseteq D_I\setminus \{i^*\}$ and $T\subseteq D_J\setminus\{j^*\}$, consider the $L^{++}$-substructure $N_{S,T}\subseteq \C^{++}$
with universe
\begin{align*}
\pbar\bigcup&\{\abar_i:i\in S \cup \set{i^*}\}\bigcup\{\bbar_j:j\in T \cup \set{j^*}\}\bigcup \\ 
&\{\cbar_{i,j}: (i,j) \in (Cil(S) \times Nb(j^*)) \cup (Nb(i^*) \times Cil(T))\}    
\end{align*}
where $\pbar$ are the interpretations of the $L^{++}$-constant symbols. 
Note that the cardinality $|N_{S,T}|$ is $O(|S|+|T|)$.  

\begin{claim}    \label{first}   For any $\dbar \subset \bigcup\set{\abar_i : i \in S}$ of length $m_a$ and any $c\in F_I\cap N_{S,T}$,
$\C^{++}\models\chi^+(\dbar,\bbar_{j^*},c)$ if and only if $N_{S,T}\models\chi^+(\dbar,\bbar_{j^*},c)$.
\end{claim}

\begin{claimproof}  As $\chi^+$ is a boolean combination of existential formulas $\delta$ in which every existential quantifier is bound to some $C^\ell$,
it suffices to prove that if $\C^{++}\models \exists w_1\in C^1\dots\exists w_{m_c-1}\in C^{m_c-1}\rho(\dbar,\bbar_{j^*},c,w_1,\dots,w_{m_c-1})$,
where $\rho$ is a quantifier-free $L^{++}$-formula, then the same holds in $N_{S,T}$.  But this is ensured by the cilia around each point of $3D_I$ and of $N(j^*)$ and by the $\P$-indiscernibility of $\A$. In particular, in $N_{S,T}$ each $c \in F_I$ is at the center of a $(2m_c+1) \times (2m_c+1)$ grid of $\cbar$-tuples arising from the cilia, and by $\P$-indiscernibility if there exist witnesses in $\C^{++}$ to the existential quantifiers above, then witnesses can be found within these grids.
\end{claimproof}

\begin{claim}  \label{second} The $L^{++}$-formula $\alpha(\xbar)$ defines $\{\abar_i:i\in S\}$ in $N_{S,T}$.
\end{claim}

\begin{claimproof}  First, for each $i\in S$, $\C^{++}\models\chi^+(\abar_i,\bbar_{j^*},\cbar^0_{k,\ell})$ if and only if $(i,j^*)=(k,\ell)$.  
Thus, by Claim~\ref{first}, $\cbar^0_{i,j^*}$ is the unique solution to $\chi^+(\abar_i,\bbar_{j^*},z)$ in $N_{S,T}$.  As $E_I(\cbar^0_{i,j^*})$ holds, we have 
$N_{S,T}\models\alpha(\abar_i)$.

Conversely, assume $N_{S,T}\models\exists z(\chi^+(\dbar,\bbar_{j^*}, z)\wedge E_I(z))$.  By construction, any witness $z$ has the form $\cbar^0_{i,j^*}$ for some $i\in S$.
Now assume $\dbar\neq\abar_i$ (and so $\dbar \neq \abar_k$ for any $k$) and we will show $N_{S,T}\models\chi^+(\dbar,\bbar_{j^*},c')$ where $c'$ is one of the eight points $\cbar^0_{i',j'}$ where $i'\in Nb(i)$, $j'\in Nb(j^*)$,
with $(i',j')\neq(i,j^*)$.  As each of these eight points is in $F_I$, we would conclude $N_{S,T}\models\neg\alpha(\dbar)$.  
To show the missing step, note that by Claim~\ref{first}, it suffices to prove this in $\C^{++}$.  The proof of \cite{MonNIP}*{Lemma 4.11} shows precisely that if $\chi^+$ satisfies the minimality condition we have imposed on it (i.e., that there are no allowable specializations), then for every $\epsilon > 0$ there is some $\cbar^0_{i',j'}$ such that $(i',j') \neq (i,j^*)$, $i - \epsilon < i < i + \epsilon$ and $j^* - \epsilon < j' < j^* + \epsilon$, and $\C^{++} \models \chi^+(\dbar, \bbar_{j^*}, \cbar^0_{i',j'})$. By $\P$-indiscernibility we may take $\cbar^0_{i',j'}$ to be one of the eight points described above.
\end{claimproof}

By symmetric claims and identical reasoning, we conclude that $\beta(\ybar)$ defines $\{\bbar_j:j\in T\}$ in $N_{S,T}$.

We are now ready to form our finite structures $M^*_G$ encoding bipartite graphs.  
Given any finite bipartite graph $G=(S,T;E)$,  let
$M^*_G$ be the $\C^{++}$ substructure with universe
$$\pbar\cup N_{S,T}\cup\{\cbar_{i,j}:(i,j)\in E\}$$
Visibly, $|M^*_G|$ is $O(|S|+|T|+|E|)$.  Since $N_{S,T}$ is definable in $M^*_G$ via $\gamma(x):=N(x)\vee x\in\pbar$, we may define $\alpha^\gamma(\xbar)$ as the ``$\gamma$-relativized'' version of $\alpha$, by taking $\alpha$ and requiring that all variables (free and bound) belong to $\gamma$, and similarly define $\beta^\gamma$. 
Now, put  $U^*(x) := A^0(x) \wedge x \neq \abar_{i^*}^0$, $V^*(y) := B^0(y) \wedge  y \neq \bbar_{j^*}^0$, and $E^*(x,y) := \exists \xbar \ybar(\alpha^\gamma(x\xbar) \wedge \beta^\gamma(y\ybar) \wedge \exists z \chi^+(x\xbar, y\ybar, z))$.
\end{proof}

\begin{theorem} \label{thm:growth}
	Let $\CC$ be a hereditary class in a relational language. If $Th(\CC)$ is not monadically NIP, then there is some $k \in \omega$ such that the unlabeled growth rate $f_\CC(n) = \Omega(\floor{n/k}!)$. 
\end{theorem}
\begin{proof}	
	If $\CC$ is not monadically NIP, let $\CC^*$ be an expansion and $U^*, V^*, E^*$ be formulas as in Proposition \ref{prop:transd}. Given a bipartite graph $G = (U,V; E)$, let $M^*_G$ be as in Proposition \ref{prop:transd}. We have $|M^*_G| = O(|U|+|V| + |E|)$, so if $G$ has $n$ edges and no isolated vertices, then $|M^*_G| \leq Kn$ for some  $K \in \omega$ depending only on $\CC$.  As shown within Case (a) of the proof of \cite[Theorem 1.5]{rapid}, the number of such graphs is $\Omega(\floor{n/5}!)$. If $G \not\cong H$ then $M^*_G \not\cong M^*_{H}$, so $f_{\CC^*}(n) = \Omega(\floor{n/5K}!)$. Having added finitely many unary predicates and named finitely many constants in passing to $\CC^*$ affects the growth rate by at most an exponential factor, so we obtain the desired bound on $f_{\CC}(n)$.
\end{proof}

\begin{remark}
	{\em The optimality of the lower bound in this theorem is witnessed by the family of hereditary classes $\set{Perm_k | k \in \Z^+}$, where $Perm_k$ encodes permutations on disjoint $k$-tuples, i.e., the language of $Perm_k$ is two $2k$-ary relations, and the structures are obtained by taking a permutation (represented as a structure in the language of two linear orders) and blowing up each point to a $k$-tuple with no further structure, and then closing under substructure. Consider a structure in $Perm_{k+1}$. By separating those points in a full $(k+1)$-tuple from those that are not, the growth rate of $Perm_{k+1}$ is seen to be bounded above by $n \floor{n/(k+1)}!$, which is $O\left(\floor{n/k}!\right)$.}
\end{remark}

\bibliographystyle{amsplain}
\bibliography{Bib}

\begin{bibdiv}
\begin{biblist}

\bib{AdAd}{article}{
      author={Adler, Hans},
      author={Adler, Isolde},
       title={Interpreting nowhere dense graph classes as a classical notion of
  model theory},
        date={2014},
     journal={European Journal of Combinatorics},
      volume={36},
       pages={322\ndash 330},
}

\bib{BS}{article}{
      author={Baldwin, John~T},
      author={Shelah, Saharon},
       title={Second-order quantifiers and the complexity of theories.},
        date={1985},
     journal={Notre Dame Journal of Formal Logic},
      volume={26},
      number={3},
       pages={229\ndash 303},
}

\bib{pwidth}{article}{
      author={Blumensath, Achim},
       title={Simple monadic theories and partition width},
        date={2011},
     journal={Mathematical Logic Quarterly},
      volume={57},
      number={4},
       pages={409\ndash 431},
}

\bib{bollobas12007hereditary}{incollection}{
      author={Bollobás, Béla},
       title={Hereditary and monotone properties of combinatorial structures},
        date={2007},
   booktitle={Surveys in {C}ombinatorics 2007},
      editor={Hilton, Anthony},
      editor={Talbot, JohnEditors},
      series={London Mathematical Society Lecture Note Series},
   publisher={Cambridge University Press},
       pages={1–40},
}

\bib{bonnet2024twin}{article}{
      author={Bonnet, {\'E}douard},
      author={Giocanti, Ugo},
      author={de~Mendez, Patrice~Ossona},
      author={Simon, Pierre},
      author={Thomass{\'e}, St{\'e}phan},
      author={Toru{\'n}czyk, Szymon},
       title={Twin-width {IV}: ordered graphs and matrices},
        date={2024},
     journal={Journal of the ACM},
      volume={71},
      number={3},
       pages={1\ndash 45},
}

\bib{braunfeld2022monadic}{article}{
      author={Braunfeld, Samuel},
       title={Monadic stability and growth rates of $\omega$-categorical
  structures},
        date={2022},
     journal={Proceedings of the London Mathematical Society},
      volume={124},
      number={3},
       pages={373\ndash 386},
}

\bib{MonNIP}{article}{
      author={Braunfeld, Samuel},
      author={Laskowski, Michael~C},
       title={Characterizations of monadic {NIP}},
        date={2021},
     journal={Transactions of the American Mathematical Society, Series B},
      volume={8},
      number={30},
       pages={948\ndash 970},
}

\bib{horizons}{article}{
      author={Braunfeld, Samuel},
      author={Ne{\v{s}}et{\v{r}}il, Jaroslav},
      author={Ossona~de Mendez, Patrice},
      author={Siebertz, Sebastian},
       title={Decomposition horizons and a characterization of stable
  hereditary classes of graphs},
        date={2025},
     journal={European Journal of Combinatorics},
      volume={129},
}

\bib{Cas}{book}{
      author={Casanovas, Enrique},
       title={Simple {T}heories and {H}yperimaginaries},
   publisher={Cambridge University Press},
        date={2011},
      number={39},
}

\bib{GRS}{book}{
      author={Graham, Ronald~L},
      author={Rothschild, Bruce~L},
      author={Spencer, Joel~H},
       title={Ramsey {T}heory},
     edition={Second {E}dition},
      series={Wiley Interscience Series in Discrete Mathematics and
  Optimization},
   publisher={John Wiley \& Sons},
        date={1990},
      volume={20},
}

\bib{hodges1993model}{book}{
      author={Hodges, Wilfrid},
       title={Model theory},
   publisher={Cambridge University Press},
        date={1993},
}

\bib{klazar2010some}{incollection}{
      author={Klazar, Martin},
       title={Some general results in combinatorial enumeration},
        date={2010},
   booktitle={Permutation {P}atterns},
      editor={Linton, Steve},
      editor={Ruškuc, Nik},
      editor={Vatter, VincentEditors},
      series={London Mathematical Society Lecture Note Series},
   publisher={Cambridge University Press},
       pages={3–40},
}

\bib{MA}{article}{
      author={Laskowski, Michael~C},
       title={Mutually algebraic structures and expansions by predicates},
        date={2013},
     journal={The Journal of Symbolic Logic},
      volume={78},
      number={1},
       pages={185\ndash 194},
}

\bib{laskowski2022jumps}{article}{
      author={Laskowski, Michael~C},
      author={Terry, Caroline~A},
       title={Jumps in speeds of hereditary properties in finite relational
  languages},
        date={2022},
     journal={Journal of Combinatorial Theory, Series B},
      volume={154},
       pages={93\ndash 135},
}

\bib{rapid}{article}{
      author={Macpherson, HD},
       title={Infinite permutation groups of rapid growth},
        date={1987},
     journal={Journal of the London Mathematical Society},
      volume={2},
      number={2},
       pages={276\ndash 286},
}

\bib{Spars}{book}{
      author={Ne{\v{s}}et{\v{r}}il, Jaroslav},
      author={De~Mendez, Patrice~Ossona},
       title={Sparsity: {G}raphs, {S}tructures, and {A}lgorithms},
   publisher={Springer Science \& Business Media},
        date={2012},
      volume={28},
}

\bib{rank}{inproceedings}{
      author={Ne{\v{s}}et{\v{r}}il, Jaroslav},
      author={Ossona~de Mendez, Patrice},
      author={Pilipczuk, Micha{\l}},
      author={Rabinovich, Roman},
      author={Siebertz, Sebastian},
       title={Rankwidth meets stability},
        date={2021},
   booktitle={Proceedings of the {T}hirty-{S}econd {A}nnual {ACM-SIAM}
  {S}ymposium on {D}iscrete {A}lgorithms},
       pages={2014\ndash 2033},
}

\bib{Noy}{incollection}{
      author={Noy, Marc},
       title={Graphs},
        date={2015},
   booktitle={Handbook of {E}numerative {C}ombinatorics},
   publisher={Chapman and Hall/CRC},
       pages={397\ndash 436},
}

\bib{pilipczuk2025graph}{article}{
      author={Pilipczuk, Micha{\l}},
       title={Graph classes through the lens of logic},
        date={2025},
     journal={arXiv preprint arXiv:2501.04166},
}

\bib{Scow}{article}{
      author={Scow, Lynn},
       title={Indiscernibles, {EM}-types, and {R}amsey classes of trees},
        date={2015},
     journal={Notre Dame Journal of Formal Logic},
      volume={56},
      number={3},
       pages={429\ndash 447},
}

\bib{Hanf}{incollection}{
      author={Shelah, Saharon},
       title={Monadic logic: Hanf numbers},
        date={1986},
   booktitle={Around {C}lassification {T}heory of {M}odels},
   publisher={Springer},
       pages={203\ndash 223},
}

\bib{shelah1990classification}{book}{
      author={Shelah, Saharon},
       title={{Classification Theory and the Number of Non-Isomorphic Models}},
   publisher={Elsevier},
        date={1990},
      volume={92},
}

\bib{Guide}{book}{
      author={Simon, Pierre},
       title={A {G}uide to {NIP} {T}heories},
   publisher={Cambridge University Press},
        date={2015},
}

\bib{SimonNote}{article}{
      author={Simon, Pierre},
       title={A note on stability and {NIP} in one variable},
        date={2021},
     journal={arXiv preprint arXiv:2103.15799},
}

\end{biblist}
\end{bibdiv}

\end{document}